\documentclass[11pt,twoside]{article}

\usepackage[english]{babel} 
\usepackage{amsmath}
\usepackage{amssymb}
\usepackage{scrextend}
\usepackage{verbatim,eufrak}
\usepackage{enumerate}
\usepackage{amsfonts}
\usepackage{cite}
\usepackage{color}
\usepackage{hyperref}

\newtheorem{proposition}{Proposition}
\newtheorem{theorem}[proposition]{Theorem}
\newtheorem{lemma}[proposition]{Lemma}
\newtheorem{corollary}[proposition]{Corollary}
\newtheorem{definition}[proposition]{Definition}
\newtheorem{remark}[proposition]{Remark}

\newtheorem{conjecture}[proposition]{Conjecture}
\newenvironment{proof}{{\noindent \it Proof.}}{\hfill $\fbox{}$ \vspace*{5mm}}

\newcommand{\vertiii}[1]{{\left\vert\kern-0.25ex\left\vert\kern-0.25ex\left\vert #1
		\right\vert\kern-0.25ex\right\vert\kern-0.25ex\right\vert}}
\newcommand{\N}{\mathbb{N}}
\newcommand{\R}{\mathbb{R}}
\newcommand{\rf}{\mathfrak{r}}
\newcommand{\Lnormal}{\textnormal{L}}
\newcommand{\Lloc}{\textnormal{L}_{\textnormal{loc}}}

\numberwithin{proposition}{section}
\numberwithin{equation}{section}
\author{
Davide Bianchi\footnote{Partially supported by GNAMPA and GNSC}\\
Dipartimento di Scienza e Alta Tecnologia,
\\		Universit\`a dell'Insubria, \\
via Valleggio 11,
		22100 Como, Italy.
\\ email: d.bianchi9@uninsubria.it
\and
Alberto G. Setti\footnote{Partially supported by GNAMPA}
\\ Dipartimento di Scienza e Alta Tecnologia\\
		Universit\`a dell'Insubria, \\
via Valleggio 11,
		22100 Como, Italy.
\\ email: alberto.setti@uninsubria.it
}
\title{Laplacian cut-offs, porous and fast diffusion on manifolds and other applications}
\begin{document}
\maketitle


\begin{abstract}
We construct exhaustion and cut-off functions with controlled gradient and Laplacian on manifolds with Ricci curvature bounded from below by a (possibly unbounded) nonpositive function of the distance from a fixed reference point,  without any assumptions on the topology or the injectivity radius. Along the way we prove a generalization of the Li-Yau gradient estimate which is of independent interest. We then apply our cut-offs to the study  of the fast and porous media diffusion, of  $L^q$-properties of the gradient and of the self-adjointness of Schr\"odinger-type operators.
\end{abstract}

\section*{Introduction.}
Many analytic results in Euclidean setting require the use of compactly supported  cut-off functions, essentially to localize differential equations or inequalities or to perform integration by parts arguments. A key feature of $d$-dimensional Euclidean space is that it is possible to construct cut-offs $\{\phi_R\}$ such that $\phi_R=1$ on the ball $B_R(o)$, they are  supported in the ball $B_{\gamma R}(o)$ and have controlled derivatives up to second order:
\[
|\nabla \phi_R| \leq \frac{C}R,\qquad |\Delta \phi_R|\leq \frac{C}{R^2}
\]
where $C$ is a constant depending only on $\gamma$ and the dimension. Indeed, such cut-offs  can be defined in terms of the distance function $r$ from $o$, $r(x)=(\sum_i x_i^2)^{1/2}$, as
\[
\phi_R(x)=\psi(r(x)/R)
\]
where $\psi:\R\to [0,1]$ is smooth, identically $1$ in $(-\infty, 1]$ and vanishes in $[\gamma, +\infty)$, and the properties of $\phi_R$ listed above depend crucially on the fact that the distance function is proper and satisfies
\[
|\nabla r(r)| \leq C, \quad\ |\Delta r| \leq \frac C{r} \quad\left(\text{indeed, }\,\, |\nabla r(r)|= 1, \quad\ \Delta r= \frac {d-1}r\right) .
\]
A proper function is often referred to as an exhaustion function, and the existence of Euclidean cut-offs with the above properties is then a consequence of the fact that distance is a well-behaved exhaustion function on $\R^d$.

While in many instances a control on the gradient of the cut-off suffices, in many other significant situations it is actually vital to have an explicit uniform decay of $\Delta \phi_R$ in terms of $R$. We quote, for example,  spectral properties of Schr\"odinger-type operators (see, e.g., \cite{leinfelder1981schrodinger}), and, most notably from our point of view, the approximation procedures used in the proof of existence, uniqueness and qualitative and quantitative properties of solutions to the Cauchy problem for the porous and fast diffusion equations (\cite{benilan1982solutions}, \cite{herrero1985cauchy}, \cite{aronson1983initial}, \cite{vazquez2007porous}), which we will take up in the second part of the paper.

It follows that the extension of such Euclidean results to the setting of Riemannian manifolds will often depend on the existence of good families of cut-offs and well behaved exhaustion functions.

While it is well  known that exhaustion functions with a control on the gradient exist under the only assumption of geodesic completeness (see \cite{gaffney1959conservation,greene1979c,shubin2001essential}), uniform bounds on the second order derivatives typically require stronger geometric assumptions. For instance,  bounded sectional curvature and a uniform strictly positive lower bound on the injectivity radius allows to construct exhaustion functions with controlled Hessian, see \cite{donnelly1997exhaustion},\cite[pg. 61]{shubin1992spectral} and \cite[Proposition 26.49]{bennettricci}.  In a very recent paper, \cite{rimoldi2016extremals}, the authors refine the arguments in \cite{bennettricci} and show that the conclusions hold assuming only that the Ricci curvature is bounded and the injectivity radius is strictly bounded away from zero.

On the other hand, it was proved in \cite[Theorem 2.2]{guneysu2016sequences} that one can construct families of cut-off functions $\{\phi_R\}$ with a Euclidean like behavior  of $|\nabla\phi_R|$ and  $|\Delta \phi_R| $ in terms of $R$, provided the Ricci curvature is nonnegative.

This paper started trying to extend the results obtained in \cite{bonforte2008fast},
by M.Bonforte, G. Grillo and L. Vazquez, where they consider Cartan-Hadamard manifolds with Ricci curvature (and therefore sectional curvture) bounded from below, under relaxed  geometric assumption. In doing so it quickly became clear that one of main tools was indeed the existence of cut-off functions with an explicit decay rate for the $|\nabla\phi_R|$ and  $|\Delta \phi_R| $.

We were thus led to investigate the existence of such cut-offs under more general geometric conditions than those considered in \cite{bonforte2008fast}, in particular, avoiding hypotheses on the injectivity radius.
The above mentioned \cite[Theorem 2.2]{guneysu2016sequences} gives a positive answer in the case of nonnegative curvature, and in \cite{schoen1994lectures1} it is shown that a good exhaustion function exists if the Ricci curvature is bounded below by a negative constant. This  suggests that this may be extended to the case manifolds with  suitable, not necessarily constant, Ricci curvature lower bounds.

A substantial part of this paper is  devoted to carry out this program and to produce both exhaustion functions and sequences of cut-offs on  manifolds whose Ricci curvature
satisfies the lower bound
$
\textnormal{Ric}\geq - (d-1) G_\alpha(r)
$
 in the sense of quadratic forms, for a family of  possibly unbounded functions  $G_\alpha$  of the distance function $r=r(x)$ from a fixed reference point $o$, and with an explicit dependence on $\alpha$ for the  bounds on the gradient and of the Laplacian.

We believe that these cut-off functions will be useful in a number of situation and the  second part of the paper is devoted to illustrating several instances, mostly coming from fast and porous media diffusions, where this is indeed the case.

The paper is organized as follows. In Section \ref{section:def} we set up notation and give the relevant definitions.

Section \ref{section:existence} is devoted to the the main technical results of the paper, the existence of $C^\infty(M)$ exhaustion functions, Theorem \ref{thm:3.1}, and of  sequences of Laplacian cut-off under generalized Ricci lower bounds, Corollary \ref{cor:3.1.2} and \ref{cor:3.1}. Their proofs  depend on several other additional results, many of independent interest, which we collect in subsection \ref{subsection:technical}. We mention in particular Theorem \ref{thm:3.2}, which generalizes  the Li-Yau gradient estimate (see \cite[Theorem 7.1]{cheeger2001degeneration}) to  functions satisfying a Poisson equation with right hand side depending both on the function itself and on $r(x)$ and under quite general Ricci curvature lower bound, and Proposition \ref{prop:3.4} which  provides  a lower bound for the volume of balls with fixed radius in terms  the distance of  their center from reference fixed point $o$, as in \cite[Proposition 4.3]{schoen1994lectures1} for manifolds satisfying suitable Ricci variable curvature lower bounds.

The last two sections are devoted to applications.

In Section \ref{section:application1} we present a first direct application of the existence of sequences of Laplacian cut-offs to obtain a generalization of the $L^q$-properties of the gradient and the self-adjointness of Schr\"odinge-type operators discussed in \cite{shubin2001essential} and  \cite{guneysu2016sequences} to
to the class of Riemannian manifolds satisfying our more general Ricci curvature conditions.

Section \ref{section:application2} is arguably the second main part of the paper. We apply the results of Section~\ref{section:existence} to study  uniqueness $\Lnormal^1$-contractivity properties and conservation of mass for the porous diffusion equation as well as uniqueness, weak conservation of mass and extinction time properties for solutions of the fast diffusion equation, which we prove under our usual quote general geometric assumptions.

\section{Basic definitions and assumptions}\label{section:def}

Throughout the paper, $(M, \langle\,, \rangle)$ is a complete noncompact $d$-dimensional Riemannian manifold, and we will often simply refer to it as $M$. We denote by $r(x):= \textnormal{dist}_M(x,o)$ the the Riemannian distance function from a fixed reference point $o\in M$.
The gradient and (negative) Laplacian of a function $u$ on  $M$ are denoted by $\nabla u$ and $\Delta $, respectively. Recall that, in local coordianates  $x^i$, they are given by
\[
\nabla u = g^{ij}\frac{\partial u}{\partial x_i}\frac{\partial}{\partial x^j}\quad
\Delta u= \frac{\partial}{\partial x^i} \left( g^{ij} \sqrt{ g} \frac{\partial u}{\partial x^j}\right)
\]
where $\{g_{ij}\}$
 is the matrix of the coefficients of the metric in the coordinates $\{x^i\}$, $\{g^{ij}\}$ its inverse and $g=\textnormal{det}\{g_{ij}\}$.

We let  $B_R(p)$  be the geodesic ball of radius $R$  centered at $p\in  M$, and with
$\partial B_R(p)$ and  $vol (B_R(p))$ its boundary and Riemannian volume. When $p=o$ we may omit the  center.

We will assume that the Ricci curvature of $M$ satisfies the inequality
\[
Ric_M (\cdot, \cdot) \geq  - (d-1) G(r),
\]
in the sense of quadratic forms where  $G(r) \in C^0([0,\infty))$.

We denote with  $M_G$ the $d$-dimensional model manifold with radial Ricci curvature equals to $-(d-1)G(r)$, namely, the manifold which is diffeomorphic to $\R^d$ and whose metric in spherical coordinates is given by
$$
\langle \cdot, \cdot \rangle_G = dr^2 + h(r)^2 d\xi^2,
$$
where $h(r)$is the solution of the problem
\begin{equation}\label{eq:3.1}
\begin{cases}
h''(r) =G(r) h(r),\\
h(0)=0,\\
h'(0)=1.
\end{cases}
\end{equation}

Let $V_G (r)$ be the volume of the ball of radius $r$ centered at the pole $o$ of $M_G$ so that
\begin{equation}\label{eq:V_G}
V_G (r)= C(d) \int_0 ^r h(t)^{d-1} dt,
\end{equation}
so that, by Laplacian comparison,
\[
\Delta r\leq (d-1)
\frac{h'(r)}{h(r)}
\]
pointwise in the complement of the cut locus of $o$ and weakly on $M$,
 and  by the Bishop-Gromov comparison theorem, for every $0\leq R_1 \leq R_2$.
\begin{equation}\label{eq:3.2}
\frac{vol (B_{R_2})(o)}{V_G(R_2)} \leq \frac{vol (B_{R_1})(o)}{V_G(R_1)},
\end{equation}

Finally, as in \cite{guneysu2016sequences}, we give the  following definition
\begin{definition}[Laplacian cut-off]\label{def:laplacian-cutoff}
$M$ admits a sequence of Laplacian cut-off functions, $\{\phi_n\}_{n\in \N} \subset C^\infty _c(M)$, if $\{\phi_n\}_{n\in \N}$ satisfies the following properties:
\begin{enumerate}
	\item $0\leq \phi_n (x) \leq 1$ for all $n \in \N$, $x \in M$;
	\item for all compact $K \subset M$ there exists $n_0(K) \in \N$ such that for every $n\geq n_0(K)$ it holds $\phi_{n|K}\equiv 1$;
	\item $\sup_{x\in M} |\nabla \phi_n (x)| \to 0$ as $n \to \infty$;
	\item  $\sup_{x\in M} |\Delta \phi_n (x)| \to 0$ as $n \to \infty$.
\end{enumerate}
\end{definition}

To indicate constants we will preferably use capital letters $A, C, D, E$, possibly with subscripts, which may change from line to line, and,  whenever necessary the dependence of the constants on the relevant parameters will be made explicit.

\section{On the existence of a sequence of Laplacian cut-off.}\label{section:existence}
In this section we collect the technical results which will allows us to prove the existence of Laplacian cut-off functions under relaxed curvature bounds. As already mentioned, we will use these cut-offs in Sections \ref{section:application1} and \ref{section:application2} below in order to extend and further generalize several different results in functional analysis and a PDE's. The main result is  Theorem \ref{thm:3.1}, where, following
the proof of \cite[Theorem 4.2]{schoen1994lectures1}, we construct $C^\infty$ exhaustion function $\rf$ whose gradient and Laplacian are  controlled in terms of explicit functions of the distance function $r$.

The key ingredients for the proof are  Theorem \ref{thm:3.2}, a generalization of Li-Yau gradient estimates which permits to obtain a control on the gradient of solutions of a Poisson equation again in terms of the distance function $r$ and the function $G$ which bounds the curvature from below, and Proposition \ref{prop:3.4} which gives a lower bound on the volume  balls with fixed radius in terms of the distance of their center from the reference point $o$.  In Corollary \ref{cor:3.1.2} we use the exhaustion function of Theorem \ref{thm:3.1} to construct a sequence $\{\phi_n\}_{n\in \N}$ of Laplacian cut-offs with support contained in a suitable increasing exhaustion of $M$. Finally, in \ref{cor:3.1}  we specialize the construction  to obtain cut-offs supported in geodesic balls and show that, when $\alpha=2$, which corresponds to an almost Euclidean situation, it is possible to construct cut-offs for which, as in Euclidean space, are equal to $1$ on a ball of radius $R>0$ and supported in a ball of radius $\gamma R$ with $\gamma>1$ arbtrarily close to $1$. This is obtained using a specific construction modelled on the proof of \cite[Theorem 6.33]{cheeger1996lower}, which basically hinges on the fact that when $\alpha=2 $ the Laplacian of the distance function satisfies $\Delta r \leq Cr^{-1}$ weakly on the whole manifold.

\begin{theorem}\label{thm:3.1}
Let $Ric_M ( \cdot , \cdot) \geq  - (d-1) \frac{\kappa^2}{(1+r^2)^{\alpha/2}} \langle \cdot , \cdot \rangle$ in the sense of quadratic forms, with $\alpha \in [-2,2]$ and $o\in M$ fixed. Then there exists an exhaustion function $\rf: M \to [0,\infty)$, $\rf \in C^\infty(M)$, and positive constants $D_{i,\alpha}$ such that
\begin{itemize}
	\item \textbf{Case $\alpha \in [-2,2)$:}
\begin{enumerate}[(1)]
	\item $D_{1,\alpha} r^{1-\alpha/2}(x) \leq \rf(x) \leq D_{2,\alpha}\max\{1; r^{1-\alpha/2}(x)\}$, for every $x \in M$,
	\item $|\nabla \rf|\leq \frac{D_{3,\alpha}}{r^{\alpha/2}}$, for every $x \in M\setminus \overline{B}_1(o)$,
	\item $|\Delta \rf|\leq \frac{D_{4,\alpha}}{r^{\alpha}}$, for every $x \in M\setminus \overline{B}_1(o)$.
\end{enumerate}	
\item \textbf{Case $\alpha=2$:}
\begin{enumerate}
	\item[(1')] $D_{1,2} \max\{1+\log(r(x)); 0\} \leq \rf(x) \leq D_{2,2}\max\{1+\log(r(x));1\}$, for every $x \in M$,
	\item[(2')] $|\nabla \rf|\leq \frac{D_{3,2}}{r}$, for every $x \in M\setminus \overline{B}_1(o)$,
	\item[(3')] $|\Delta \rf|\leq \frac{D_{4,2}}{r^{2}}$, for every $x \in M\setminus \overline{B}_1(o)$.
\end{enumerate}	
\end{itemize}
\end{theorem}
\begin{proof}
Let us observe that $r(x)$ is not necessarily smooth everywhere but it is Lipschitz on all of $M$ with uniform unitary Lipschitz constant and then it is possible to uniformly approximate $r(x)$ by a smooth function $r_\epsilon(x)$ such that $|r_\epsilon (x) - r(x)|\leq \epsilon$ and $|\nabla r_\epsilon(x)| \leq 1+\epsilon$ for every $x \in M$ and $\epsilon>0$ fixed, see \cite[Section 2]{greene1979c}, which is enough for our purpose since every ball with respect to the Riemannian distance $r(x)$ contains and is contained by a ball with respect to the approximating function $r_\epsilon(x)$. Thus, without loss of generality, hereafter we will consider $r(x)$ to be $C^\infty$ on $M\setminus\{p\}$.

We first prove the Theorem for $\alpha \in [0,2)$.
Let $\omega_R : B_{R}(o)\setminus \overline{B}_{1/2} (o) \to [0,1]$ be such that
		$$
		\begin{cases}
		\Delta \omega_R (x)= \frac{A_1^2 C^2}{r^\alpha(x)} \omega_R (x), \\
		\omega_{R| \partial B_{1/2}} \equiv 1, \\
		\omega_{R| \partial B_{ R}} \equiv 0,
		\end{cases}
		$$
		where $A_1 =(1-\alpha/2)/\sqrt{2}$ and $C>0$ is a constant that is chosen like in Remark \ref{rem2}. By the maximum principle, $\{\omega_{R_n} \}$ is an increasing and bounded family of functions for every $x \in M\setminus \overline{B}_{1/2}(o)$ as $R_n \to \infty$, and therefore there exists the point-wise limit function
\[
\omega(x) := \lim_{R \to \infty}\omega_R (x).
\]
By $L^p$ and Schauder estimates, there exists a subsequence which converges in $C^\infty(\overline{B}_{R_n}\setminus B_{1/2})$ for every $n$, so that  $\omega \in C^\infty(M\setminus \overline{B}_{1/2})$ and
\[
		 	\begin{cases}
		 	\Delta \omega (x)= \frac{A_1^2 C^2}{r^\alpha(x)} \omega (x), \\
		 	\omega_{| \partial B_{1/2}} \equiv 1, \\
		 	0<\omega <1 & \textnormal{on } M\setminus \overline{B}_{1/2}(o).
		 	\end{cases}
\]
Integrating by parts and using $\omega_{R| \partial B_{R}} =0$, we get
\begin{align}\label{eq:3.21}
		\int_{B_{R} \setminus B_{1/2}} e^{C r^{1-\alpha/2}(x)}\omega_R \Delta \omega_R &=\int_{\partial B_{1/2}} \omega_R \frac{\partial \omega_R}{\partial \eta} - \int_{B_{R} \setminus B_{1/2}} \langle \nabla ( e^{C r(x)^{1-\alpha/2}}\omega_R), \nabla \omega_R \rangle \\
		&=\int_{\partial B_{1/2}}  \frac{\partial \omega_R}{\partial \eta} - \int_{B_{R} \setminus B_{1/2}} \left(\frac{(1-\alpha/2)C}{r^{\alpha/2}(x)}\right)  e^{C r^{1-\alpha/2}(x)} \omega_R \langle \nabla r(x) , \nabla \omega_R \rangle \nonumber\\
		&\;\;\;\;  - \int_{B_{R} \setminus B_{1/2}}  e^{C r^{1-\alpha/2}(x)} |\nabla \omega_R|^2\nonumber\\
		&\leq A_2   - \int_{B_{R} \setminus B_{1/2}} \left(\frac{(1-\alpha/2)C}{r^{\alpha/2}(x)}\right)  e^{C r^{1-\alpha/2}(x)} \omega_R \langle \nabla r(x) , \nabla \omega_R \rangle \nonumber\\
		& \;\;\;\; - \int_{B_{R} \setminus B_{1/2}}  e^{C r^{1-\alpha/2}(x)} |\nabla \omega_R|^2, \nonumber
\end{align}
		where the constant $A_2$ is independent of $R$ by elliptic estimates, since $\omega_R$ is uniformly bounded for every $R$ and $\omega_{R|\partial B_{1/2}} \equiv 1$.
		Then,		
		\begin{align*}
		\int_{B_{R} \setminus B_{1/2}} \frac{A_1^2C^2}{r^\alpha(x)} e^{C r^{1-\alpha/2}(x)g(r(x))}\omega_R^2 &\leq A_2   - \int_{B_{R} \setminus B_{1/2}} \left(\frac{(1-\alpha/2)C}{r^{\alpha/2}(x)}\right)  e^{C r^{1-\alpha/2}(x)} \omega_R \langle \nabla r(x) , \nabla \omega_R \rangle \\
		&\;\;\;\; - \int_{B_{R} \setminus B_{1/2}}  e^{C r^{1-\alpha/2}(x)} |\nabla \omega_R|^2  \\
		&\leq A_2 +  \int_{B_{R} \setminus B_{1/2}}  \left(\frac{(1-\alpha/2)C}{r^{\alpha/2}(x)}\right)  e^{C r^{1-\alpha/2}(x)} \omega_R | \nabla \omega_R |\\
		& \;\;\;\; - \int_{B_{R} \setminus B_{1/2}}  e^{C r^{1-\alpha/2}(x)g(r(x))} |\nabla \omega_R|^2\\
		& \leq A_2 + \int_{B_{R} \setminus B_{1/2}} \left( \frac{(1-\alpha/2)^2C^2}{4r^{\alpha}(x)}\omega_R ^2 + |\nabla \omega_R|^2  \right)e^{C r^{1-\alpha/2}(x)} \\
		&\;\;\;\; - \int_{B_{R} \setminus B_{1/2}}  e^{C r^{1-\alpha/2}(x)} |\nabla \omega_R|^2 \\
		&= A_2 +  \int_{B_{R} \setminus B_{1/2}}\frac{(1-\alpha/2)^2C^2}{4r^{\alpha}(x)}  e^{C r^{1-\alpha/2}(x)} \omega_R ^2.
		\end{align*}
		It follows that
		$$
		\int_{B_{ R} \setminus B_{1/2}} \frac{A_1^2C^2}{2r^{\alpha}(x)} e^{C r^{1-\alpha/2}(x)} \omega_R ^2 \leq A_2,
		$$
		and then, by letting $R \to \infty$,
		$$
		\int_{M \setminus B_{1/2}} \frac{A_1^2C^2}{2r^{\alpha}(x)} e^{C r^{1-\alpha/2}(x)} \omega ^2 \leq A_2.
		$$
		Let $x \in M \setminus \overline{B}_{1}$ and $y \in B_{1/4} (x) \subset \left( M\setminus \overline{B}_{3/4}\right)\subset \left( M\setminus \overline{B}_{1/2}\right)$. By the triangle inequality,
		$$
		 r(x) - 1/4 \leq  r(y) \leq  r(x) + 1/4
		$$
		and then
		\begin{align*}
		& \frac{A_1^2C^2e^{C (r(x)-1/4)^{1-\alpha/2}} }{2(r(x)+1/4)^{\alpha}}\int_{B_{1/4}(x)} \omega^2 (y) \leq \int_{B_{1/4}(x)} \frac{A_1^2C^2}{2r^{\alpha}(y)} e^{C r^{1-\alpha/2}(y)} \omega ^2 (y) \\
		&\leq \int_{M \setminus B_{1/2}(p)} \frac{A_1^2C^2}{2r^{\alpha}(x)}  e^{C r^{1-\alpha/2}(x)} \omega ^2  \leq A_2 ,
		\end{align*}
		namely
		\begin{align*}
		\int_{B_{1/2}(x)} \omega^2 (y) &\leq  \frac{2A_2(r(x)+1/4)^{\alpha}}{A_1^2C^2}  e^{-C (r(x)-1/4)^{1-\alpha/2}} \\
		&\leq \frac{2^{\alpha+1}A_2r^{\alpha}(x)}{A_1^2C^2}  e^{-2^{-(1-\alpha/2)}C r^{1-\alpha/2}(x)}.
		\end{align*}
		By Theorem \ref{thm:3.2} and Corollary \ref{cor:3.2} applied with
		\begin{align*}
		&\zeta=r, \;\; f_1(r) = \frac{A_1^2 C^2}{r^\alpha}, \;\; f_2(\omega)=\omega,\\
		&G(r) = \frac{\kappa^2}{(1+ r^2)^{\alpha/2}},
		&R_0=1/2, R_1=3/4, t=1/8, \gamma=\infty,
		\end{align*}
		we deduce that
		$$
		|\nabla \log \omega (y) | \leq C_3 (d, \kappa,\alpha)
		$$
		on $M \setminus \overline{B}_{3/4} \supset B_{1/4} (x)$, and then, letting $\sigma$, be a geodesic parametrized by arc length connecting $x$ to $y$, from the path integral
		$$
		|\log \omega(x) - \log \omega(y)| = \int_{0}^{\textnormal{dist}(x,y)} |\nabla \log(\sigma (s)) | ds \leq \frac{C_3}{4},
		$$
		we infer that $\omega(y) \geq e^{-C_3/4} \omega(x)$, and that implies
		$$
		e^{-C_3/2} Vol (B_{1/4} (x)) \omega^2(x) \leq  \frac{2^{\alpha+1}A_2r^{\alpha}(x)}{A_1^2C^2}  e^{-2^{-(1-\alpha/2)}C r^{1-\alpha/2}(x)},
		$$
		namely
		\begin{equation*}\label{eq:pr1}
		\omega^2 (x) \leq  \frac{2^{\alpha+1}A_2r^{\alpha}(x)e^{C_3/2}}{A_1^2C^2}   \left( vol (B_{1/4} (x)) \right)^{-1} e^{-2^{-(1-\alpha/2)}C r^{1-\alpha/2}(x)}.
		\end{equation*}
		By Proposition \ref{prop:3.4}, and by the fact that $r^{\alpha/2} \leq A_3e^{A_3r^{1-\alpha/2}}$, we conclude that
		\begin{equation}\label{eq:1.1dis}
		\omega (x) \leq C_4  e^{-C_5 r^{1-\alpha/2}(x)}, \qquad \textnormal{on } M \setminus \overline{B}_{1}(o),
		\end{equation}
		with
		\begin{equation}\label{C_5C_4}
		C_4 = \sqrt{\frac{2^{\alpha+1}A_2A_3e^{C_3/2}}{A_1^2C^2 \bar{C}_1}}, \qquad  C_5 = \frac{2^{-(1-\alpha/2)}C - \bar{C}_2}{2} -A_3,
		\end{equation}
		and where $\bar{C}_1$ and $\bar{C}_2$ are the constant that appears in the statement of Proposition \ref{prop:3.4}.
		
		Extend now $\omega(x)$ on all of $M$ fixing $\omega(x)\equiv 1$ for every $x \in B_{1/2}(p)$ and define
		$$
		\rf(x) := (\eta (x) -1)\log(\omega(x)) + \eta(x),
		$$
		with $\eta \in C^\infty$, $\eta(x) \equiv 1$ on $B_{1/2}(p)$ and $\eta(x) =0$ on $M\setminus \overline{B}_1 (o)$. Observe that $0<E_1\leq h(x)\leq E_2$ on $\overline{B}_1(p)$ and in particular $\rf(x) \geq E_1r^{1-\alpha/2}$ on $\overline{B}_1(o)$.
		
		Fix $x \in M \setminus \overline{B}_{1}(o)$ and let $\sigma_o :[0, r(x)] \to M$ be a geodesic parametrized by arc length joining $o$ and $x$. Then
		\begin{align}
		|\rf(x) - \rf(p)|= \int_0 ^{r(x)} |\nabla \rf(\sigma_o(s))| ds &= \int_{1/2} ^{r(x)} |\nabla \rf(\sigma_o(s))| ds \nonumber\\
		&= \int_{1/2}^1 |\nabla \rf(\sigma_o(s))| ds + \int_1 ^{r(x)}|\nabla \rf(\sigma_o(s))| ds \nonumber\\
		&\leq C_6 + \int_{1}^{r(x)}\frac{C_7}{s^{\alpha/2}} ds \nonumber\\
		&\leq C_8 r^{1- \alpha/2}(x),\label{eq:3.22}
		\end{align}
		where we used again Theorem \ref{thm:3.2} and Corollary \ref{cor:3.2} applied with
		\begin{align}
		&\zeta=r, \;\;f_1(r) = \frac{A_1^2 C^2}{r^\alpha}, \;\; f_2(\omega)=\omega, \label{eq:ap1Thm3.2}\\
		&G(r) = \frac{\kappa^2}{(1+ r^2)^{\alpha/2}},
		&R_0=1/2, R_1= s, t=t(s)\equiv\frac{1}{4}, \gamma=\infty, \label{eq:ap2Thm3.2}
		\end{align}
		with $t$ chosen in such a way that $(1-t)R_1=(1-t)s >1/2=R_0$, uniformly for every $s>1$. Observe that $C_6$ can be chosen independent of $x$. Henceforth,
		\begin{equation}\label{eq:pr2.1}
		\rf(x) \leq \max\{E_2; (1+ C_8)r^{1-\alpha/2}\} \quad \textnormal{on } M.
		\end{equation}
		On the other hand, since $C_5 - \log(C_4) >0$ by Remark \ref{rem2},  from inequality \eqref{eq:1.1dis} we have that
		 \begin{align}
		 \rf(x) = -\log(\omega (x)) \geq C_5r^{1-\alpha/2} - \log(C_4) &\geq (C_5 - \log(C_4))r^{1-\alpha/2} \quad \textnormal{on } M\setminus \overline{B}_1(o)\nonumber\\
		 &\geq \min\{E_1; C_5 - \log(C_4)\} r^{1-\alpha/2} \quad \textnormal{on } M,\label{eq:pr2.2}
		 \end{align}
		and putting together the above inequality \eqref{eq:pr2.2} with \eqref{eq:pr2.1} we conclude that
		$$
		D_{1,\alpha} r^{1-\alpha/2}(x) \leq \rf(x) \leq D_{2,\alpha}\max\{1; r^{1-\alpha/2}(x)\} \quad \textnormal{for every } x \in M.
		$$		
		Finally, for every $x \in M\setminus \overline{B}_1(o)$, we have
		\begin{itemize}
			\item [(i)] $|\nabla \rf| = \left| \frac{\nabla \omega}{\omega} \right|$,
			\item [(ii)] $|\Delta \rf| \leq \left| \frac{\nabla \omega}{\omega}\right|^2 + \frac{\Delta \omega}{\omega} =  \left| \frac{\nabla \omega}{\omega}\right|^2 + \frac{A_1^2 C^2}{r^\alpha (x)}$,
		\end{itemize}
		and the last statements of the thesis follow one more time by an application of Theorem \ref{thm:3.2} and Corollary \ref{cor:3.2} with \eqref{eq:ap1Thm3.2} and \eqref{eq:ap2Thm3.2}.
		
To conclude, the cases $\alpha\in [-2,0)$ and $\alpha =2$  can be proven with suitable modifications of the previous proof. Indeed, observe that for $\alpha \in [0,2)$ we used crucially the lower bound estimate for the volume of ball of fixed radius, $vol\left(B_{1/4}(x)\right)$, that appears in Proposition \ref{prop:3.4}. Therefore, replacing the exponential function $e^{Cr^{1-\alpha/2}(x)}$ in the integral \eqref{eq:3.21} with $r^{C[1+ (d-1)(1+\sqrt{1+4\kappa^2})]}(x)$, where $C>0$ is chosen big enough, will do the trick for the case $\alpha =2$, for example. For the case $\alpha \in [-2,0)$ we need one more remark: the constant $C_3$ has to be replaced by $\tilde{C}_3(r)= C_3(d,\kappa,\alpha)r^{-\alpha/2}(x)$. The estimates that  follow still hold with suitable changes. In \ref{C_5C_4} we  have
$$
\tilde{C}_4(r) = C_4 e^{\frac{r^{-\alpha/2}}{4}}\leq C_4e^{r^{1-\alpha/2}},
$$
and then choosing $C$ big enough such that $C_5 -1 >0$ we  still recover an upper bound for $\omega(x)$ of the form of \eqref{eq:1.1dis}. For the lower bound \eqref{eq:pr2.1} instead, the estimate comes directly from \eqref{eq:3.22} where now $\alpha \in [-2,0)$.             		 	
\end{proof}

\begin{remark}[On the choice of the constant $C$ in the proof of Theorem \ref{thm:3.1}.]\label{rem2}
If $C_4$ and $C_5$ are defined as in \eqref{C_5C_4}, we choose $C$ big enough such that $C_5 >0$ and $C_5 - \log(C_4)>0$. We want to stress that all  constants that appear in the definition of $C_4$ and $C_5$ are independent of the radius $R$ and consequently this independence carries over to $C$ as well.
\end{remark}
Using the exhaustion function of Theorem \ref{thm:3.1}, it is easy to construct sequences of cut-off function with explicitly controlled gradient and Laplacian. In the almost Euclidean case where $\alpha=2$, we actually use a construction inspired by \cite{cheeger2001degeneration} which relies on the fact that the Laplacian of the distance function satisfies the weak inequality $\Delta r\leq C r^{-1}$ globally on $M$, and allows to construct cut-offs which are $1$ on the  ball of radius  $R$ and vanish off in a ball of radius $\gamma R$ with $\gamma$ arbitrarily close to $1$.
\begin{corollary}\label{cor:3.1}
Let $Ric_M ( \cdot , \cdot)$ be as in Theorem \ref{thm:3.1}. Then, for every $R\geq 1$ when $\alpha \in [-2,2)$, $R>0$ when $\alpha=2$, and
$$
\gamma > \Gamma(\alpha, \kappa, d) \geq \begin{cases} \frac{D_{2,\alpha}}{D_{1,\alpha}} \geq 1 & \textnormal{for } \alpha \in [-2,2),\\
1 &  \textnormal{for } \alpha=2,
\end{cases}
$$
there exist $\phi : M \to [0,1]$, $\phi \in C^\infty _c(M)$, such that
\begin{enumerate}[(i)]
\item $\phi_{|B_R (p)} \equiv 1$,
\item supp$(\phi) \subset B_{\gamma R} (o)$,
\item $| \nabla \phi | \leq \frac{C_1}{R}$,
\item $| \Delta \phi | \leq \frac{C_2}{R^{1+\alpha/2}}$,
\end{enumerate}
with $C_1, C_2$ independent of $R$. If we choose $R=n$ then we have a sequence of Laplacian cut-offs $\{\phi_n\}_{n\in \mathbb{N}}$   with respect to the metric balls in the sense of Definition \ref{def:laplacian-cutoff}.  

\end{corollary}
\begin{proof}
\begin{itemize}
		\item \textbf{Case $\alpha \in [-2,2)$.}
\end{itemize}		
Let $\rf$, $D_1$ and $D_2$ be the function and the constants that appear in the statement of the preceding Theorem, respectively. Define $\Gamma = \frac{D_2}{D_1}$, let $\gamma > \Gamma$ be fixed and let $\psi : \R \to [0,1]$ be such that
\begin{itemize}
	\item [(i)] $\psi (r) \equiv 1$ for $r \leq \frac{D_2}{D_1}$, $0\leq \psi \leq 1$;
	\item[(ii)] $\textnormal{sup }\psi \subset (-\infty, \theta)$;
	\item[(iii)] $\psi \in C^\infty$ and $|\psi'| + |\psi''| \leq A_1$.
\end{itemize}
 Then, the function defined by
 $$
 \phi(x) := \psi\left(\frac{\rf(x)}{D_1 R^{1-\alpha/2}}\right),
 $$
is a cut-off with the desired properties.

\begin{itemize}
\item \textbf{Case $\alpha =2$.}
\end{itemize}
In order to get a better estimate of the constant $\Gamma$, for this case we will not use the exhaustion function of Theorem \ref{thm:3.1}.

Define $a= (d-1)\frac{1+\sqrt{1+4\kappa^2}}{2}$ as in Lemma \ref{lem:3.1.1} and fix $\gamma>1$. Then there exists a function $u~:~(0, + \infty) \to \R$ such that
\begin{enumerate}
	\item  $u \in C^\infty ((0, +\infty))$ and $u'' (r) + \frac{a}{r}u'(r) = \frac{1}{\gamma^{a+1} R^2 }$,
	\item  $u' (r) < 0$ on $[R, \gamma R]$,
	\item $u(R)=1$ and $u(\gamma R)$=0.
\end{enumerate}
Observe that $u$ is precisely the function defined in \eqref{eq:3.4} with the constant $C_2$ given in  \eqref{eq:3.3}.

Now let $\omega : \overline{B}_{\gamma R} (o) \setminus B_R (o) \to \R$ satisfy
$$
\begin{cases}
\Delta \omega = \frac{1}{\gamma^{a+1} R^2 },\qquad \textnormal{on } B_{\gamma R}(o)\setminus B_R (o), \\
\omega_{| \partial B_R} \equiv 1, \\
\omega_{|\partial B_{\gamma R}} \equiv 0,
\end{cases}
$$
By Proposition \ref{prop:3.1} and Lemma \ref{lem:3.3}, $u$ satisfies the weak inequality
$$
\Delta u(r) \geq \frac{1}{\gamma^{a+1} R^2 }
$$
and, applying the minimum principle to $\omega - u$, we have that
\begin{equation}\label{eq:1.2dis}
\omega \geq u \qquad \mbox{on } \overline{B}_{\gamma R} (o) \setminus B_R (o).
\end{equation}
Next let $x \in B_{\gamma R}(o)\setminus B_R(o)$. Then, for every $y \in B_{\frac{R}{2}}(x)$
$$
r(y) \geq r(x) - \frac{R}{2} \geq \frac{R}{2} \geq  \textnormal{dist}_M(x,y)=s(y),
$$
and for every $y \in B_{\frac{R}{2}}(x)$
\begin{equation*}
Ric_M (\nabla s(y), \nabla s(y)) \geq -\frac{(d-1)\kappa^2}{1+ r^2(y)}\geq   -\frac{(d-1)\kappa^2}{1+ s^2(y)}
\end{equation*}
and therefore
$$
\Delta s(y) \leq \frac{a}{s} \qquad \textnormal{in } B_{\frac{R}{2}}(x).
$$
Next consider the problem
$$
\begin{cases}
v''(s) + \frac{a}{s} v'(s) = \frac{1}{\gamma^{a+1}R^2}\\
v'(s) >0,\\
v(0)=0,
\end{cases}
$$
whose solution is
\begin{equation}
v(s) = As^2,
\end{equation}
with
$$
2A + 2aA = \frac{1}{\gamma^{a+1}R^2},
$$
namely
$$
A= \frac{1}{2(a+1)\gamma^{a+1}R^2},
$$
for which
\begin{equation}
v\left(\frac{R}{2}\right) = \frac{1}{8(a+1)\gamma^{a+1}R^2} <1.
\end{equation}
It follows that $v(s)$ satisfies
\begin{itemize}
	\item[(i)] $\Delta v(s) \leq   \frac{1}{\gamma^{a+1}R^2}$ in $B_{\frac{R}{2}}(x),$
	\item[(ii)] $v(s) = \frac{(\gamma -1)^2}{2(a+1)\gamma^{a+1}R^2} <1$ on $\partial B_{\frac{R}{2}}(x)$.
\end{itemize}
Let now $\omega : \overline{B}_{\gamma R} (o) \setminus B_R (o) \to \R$ be a function that satisfies
$$
\begin{cases}
\Delta \omega = \frac{1}{\gamma^{a+1}R^2} \qquad \textnormal{on } B_{\gamma R}(o)\setminus B_R (o),\\
\omega_{|\partial B_{R} (p)} \equiv 1,\\
\omega_{|\partial B_{\gamma R} (p)} \equiv 0.
\end{cases}
$$

Similarly, if $x \in B_{\gamma R}(o)\setminus B_R (o)$ then the function $v(y)= v(s(y))$ (where $s(y) = \textnormal{dist}_M(x,y)$) satisfies
$$
\Delta v(s) \leq \frac{1}{\gamma^{a+1} R^2 } \qquad \text{weakly},
$$
and then
\begin{equation}
\Delta \left(\omega(y) - v(s(y)) \right) \geq 0 \qquad \mbox{for every } y \in \Omega=  B_{\gamma R}(o)\setminus \overline{B}_R (o)\cap B_{\frac{R}{2}}(x).
\end{equation}
Setting
$$
\partial \Omega_1 = \overline{B}_{\frac{R}{2}}(x) \cap\partial B_R (o), \qquad \partial \Omega_2 =  \overline{B}_{\frac{R}{2}}(x) \cap \partial B_{\gamma R}(o), \qquad \partial \Omega_3 = \partial \Omega \setminus (\partial \Omega_1 \cup \partial \Omega_2),
$$
then, by the maximum principle, we have that
$$
\omega(y) -v(s(y)) \leq \max\left\{ \left[\omega(y) -v(s(y))\right]_{| \partial \Omega_1} ,  \left[\omega(y) -v(s(y))\right]_{| \partial \Omega_2}, \left[ \omega(y) - v(s(y))\right]_{| \partial \Omega_3} \right\}
$$
for every $y \in \Omega$, and using the fact that $s(y) \geq |r(x) - r(y)|$, it follows that
\begin{align*}
\omega(y) -v(s(y)) &\leq\max\left\{ 1 - v (s(y))_{| \partial \Omega_1};\; - v(s(y))_{| \partial \Omega_2};\; \omega(y)_{| \partial \Omega_3} - v(r(x)) \right\}\\
&\leq \max\left\{ 1 - v (|r(x) - r(y)|)_{| \partial \Omega_1};\; 0 \; ;\; 1 -v\left(\frac{R}{2}\right) \right\}\\
&\leq \max\left\{ 1 - v (r(x) - R);\; 0\;;\; 1 - v\left(\frac{R}{2}\right) \right\}\\
&\leq \max\left\{ 1 - v\left(\frac{r(x) - R}{2(\gamma-1)}\right);\; 0 \; ; \;1- v\left(\frac{R}{2}\right)\right\}\\
&= 1 - v\left(\frac{r(x) - R}{2(\gamma-1)}\right).
\end{align*}
Since $v(0)=0$, evaluating at $y=x$, we get
\begin{equation}\label{eq:2dis}
\omega(x) \leq  1 -v\left(\frac{r(x) - R}{2(\gamma-1)}\right)  \qquad \mbox{for every } x \in B_{\gamma R} (o) \setminus \overline{B}_R (o).
\end{equation}
Combining \eqref{eq:1.2dis} with \eqref{eq:2dis} we have that
$$
u(x) \leq \omega(x) \leq   1 - v\left(\frac{r(x) - R}{2(\gamma-1)}\right)  \qquad \mbox{for every } x \in B_{\gamma R} (o) \setminus \overline{B}_R (o).
$$
For $\theta \in  [0, \frac{\gamma -1}{2}]$ define
\begin{align*}
h_R(\theta) &= u ( (1 + \theta) R) - 1 + v\left(\frac{ (\gamma-1-\theta)R}{2(\gamma-1)}\right)\\
&=  \frac{1 + \frac{\gamma^2 -1}{2\gamma^{a+1}(a +1)}}{1- \gamma^{1-a}} \left[ (\theta +1)^{1-a} -1  \right] + \frac{(\theta +1)^2 -1}{2\gamma^{a+1} (a+1)} + \frac{(\gamma -1 - \theta)^2}{8\gamma^{a+1}(\gamma -1)^2(a+1)}.
\end{align*}
Then  $h_R(\theta) = h(\theta)$ is independent of $R$, monotone decreasing, and, since $h(0)= \frac{(\gamma-1)^2}{2\gamma^{a+1}(2\gamma -1)^2(a+1)} \in (0,1)$, there exists $\theta = \theta(d, \kappa,\gamma) \in (0, \frac{\gamma -1}{2})$ independent of $R$ such that
\begin{equation}\label{eq:3.5}
0 <  \frac{(\gamma-1)^2}{16\gamma^{a+1}(\gamma -1)^2(a+1)} \leq  u((1 + \theta) R) - 1 + v\left(\frac{r(x) - R}{2(\gamma-1)}\right) <1 .
\end{equation}
Finally, let $\psi : [0,1] \to [0,1]$ satisfy
\begin{enumerate}
	\item $\psi_{|[u((1 + \theta) R), 1]} \equiv 1$;
	\item $\psi_{|[0, 1 - v((\gamma -1-\theta)R/(\gamma-1))]} \equiv 0$;
	\item $\psi \in C^\infty([0,1])$ and $|\psi'| + |\psi''| \leq C$, with $C= C(d, \kappa, \gamma)$ independent of $R$ by \eqref{eq:3.5},
\end{enumerate}
and define
\[
\phi = \psi \circ \omega.
\]
Recalling that $u(r(x))\leq \omega(x) \leq 1 -v\left(\frac{ (\gamma-1-\theta)R}{2(\gamma-1)}\right)$, we have that
\begin{enumerate}
	\item $\phi_{|(\overline{B}_{(1 + \theta) R} \setminus B_R)}(x) \equiv 1$,
	\item $\phi_{|(\overline{B}_{\gamma R} \setminus B_{(\gamma  - \theta) R})}(x) \equiv 0$,
	\item $\nabla \phi = \psi' \nabla \omega$,
	\item $\Delta \phi = \psi'' |\nabla \omega |^2 + \frac{\psi'}{\gamma^{a+1} R^2}$.
\end{enumerate}
We extend $\psi$ to all of $M$ by setting it equal to $1$ in $B_R$, and note that, since
$$
0<  \frac{\left(1 + \frac{\gamma^2 -1}{2(a+1)\gamma^{a+1}}\right) \left[(\gamma -\theta/2)-1 \right]}{1- \gamma^{1-a}} +1 + \frac{(\gamma -\theta/2)^2 -1}{2\gamma^{a+1}(a+1)}=  u((\gamma  - \theta/2) R)  \leq \omega
$$
on $\overline{B}_{((\gamma  - \theta/2) R)} \setminus B_{((1 + \theta/2) R)}$ independently from $R$, the required conclusion follows from Theorem \ref{thm:3.2}, Remark \ref{rem:3} and Corollary \ref{cor:3.2} applied with $\alpha =2$, $\zeta=r$, $f_1(r)\equiv 1$ and $f_2(\omega)~\equiv~\frac{1}{\gamma^{a+1} R^2}$.
\end{proof}

In some approximation procedures used in the theory of  diffusion, one needs to have sequences of  cut-off functions whose zero level sets are compact smooth submanifolds. This is addressed in the next corollary.
\begin{corollary}\label{cor:3.1.2}
Let $\textnormal{Ric}(\cdot, \cdot)$ be as in Theorem \ref{thm:3.1}. Then, for every $\alpha \in (-2,2]$ there exists an increasing  exhaustion  of $M$ by open relatively compact sets $\{F_n\}_{n \in \mathbb{N}}M$ with smooth boundary with $\bar{F}_n \subset F_{n+1}$, and a sequence of functions, $\{\phi_n\}_{n\in \mathbb{N}}\subset C_c^\infty (M)$, such that
\begin{enumerate}
	\item $\phi_n \equiv 1$ on $F_n$;
	\item $0<\phi_n <1$ on $F_{n+1}\setminus \overline{F}_n$;
	\item $\phi_n \equiv 0$ on $\partial F_{n+1}$ and $\textnormal{supp}(\phi_n) = \overline{F}_{n+1}$;
	\item $\sup_x|\nabla \phi_n(x)| \to 0,$ as $n \to \infty$;
	\item $\sup_x|\Delta \phi_n(x)| \to 0,$ as $n \to \infty$.
\end{enumerate}
The sequence $\{\phi_n\}_{n\in \mathbb{N}}$ is a Laplacian cut-off  in the senso of Definition \ref{def:laplacian-cutoff}.
\end{corollary}
\begin{proof}
Let  $\rf$ be exhaustion function constructed in Theorem \ref{thm:3.1}, and let $\alpha \in (-2,2)$. Using (1) in the statement of Theorem~\ref{thm:3.1}, we may write (2) and (3) in the form
$$
|\nabla \rf|\leq \frac{C_1}{\rf^{\frac{\alpha/2}{1-\alpha/2}}}, \qquad |\Delta \rf|\leq \frac{C_2}{\rf^{\frac{\alpha}{1-\alpha/2}}},
$$
on $M\setminus \overline{B}_1 (o)$. Since $\rf \in C^\infty(M)$, by Sard's theorem we can chose a sequence $c_n$ of regular values of $\rf$ such that $|\frac{c_{n+1}}{c_{n}} -2|\leq 1/n$. Let $F_n := \{x\in M : \rf(x) < c_n \}$.  Then
$\{F_n\}_{n\in \mathbb{N}}$ is an exhaustion of $M$ by relatively compact open sets with smooth boundary, such that   $\bar{F}_n \subset F_{n+1}$. For every $n$, $\psi_n : \mathbb{R} \to [0,1]$ be a smooth real function such that
\begin{enumerate}[(a)]
	\item $\psi_n \equiv 1$ on $(-\infty, c_n]$;
	\item $0<\psi_n<1$ on $(c_n, c_{n+1})$;
	\item $\psi_n \equiv 0$ on $[c_{n+1}, +\infty)$;
	\item $|\psi'(s)|\leq \frac{A_1}{c_n}$, $|\psi''(s)|\leq \frac{A_2}{c_n^2}$.
\end{enumerate}
Then, $\phi_n := \psi_n \circ h$ satisfies the requirements. In particular,
\begin{align*}
&|\nabla \psi_n (x)| = |\psi_n'(h(x))||\nabla \rf(x)| \leq \frac{D_1}{n^{\frac{1}{1-\alpha/2}}} \xrightarrow{n\to \infty} 0  \textnormal{ for every }\alpha \in [-2,2),\\
&|\Delta \psi_n (x)| \leq |\psi_n''(h(x))||\nabla \rf(x)|^2 + |\psi_n'(h(x))||\Delta h(x)|  \leq \frac{D_1}{n^{\frac{1+\alpha/2}{1-\alpha/2}}}  \xrightarrow{n\to \infty} 0 \textnormal{ for every }\alpha \in (-2,2).
\end{align*}
The case $\alpha =2$ is dealt similarly with small changes in the proof.
\end{proof}
\subsection{Auxiliary results.}\label{subsection:technical}
In this subsection we collect some results which we used in above constructions. The first one is an extension of the classical gradient Li-Yau estimate which we establish, under rather general  Ricci curvature lower bounds, for  solutions of Poisson equations with right hand side depending both on the function itself and on the point on the manifold (via an approximate distance function). We belive that this result is of independent interest.
\begin{theorem}\label{thm:3.2}
Let $\textnormal{Ric}_M( \cdot , \cdot) \geq  - (d-1)G(r) \langle \cdot , \cdot \rangle$ on $M$ in the sense of quadratic forms, where,  $r=r(x)$ is the distance function from a fixed point $o \in M$. \\
Let $R_1 > R_0 >0$, $\gamma>1$ and let $\omega : M \setminus \overline{B}_{R_0}(o)  \to \R$ be a $C^2$ function satisfying
\begin{equation}\label{eq:3.14}
\begin{cases}
\omega > 0 & \textnormal{on } M \setminus \overline{B}_{R_0}(o), \\
\Delta \omega = f_1(\zeta)f_2(\omega),
\end{cases}
\end{equation}
where $f_1, f_2 : [0, +\infty) \to \R$ are $C^1$ functions and $\zeta: M \to [0,+\infty)$ is such that $|\nabla \zeta(x)|\leq L$ for every $x \in M$. Moreover, fix $t>0$ such that $(1-t)R_1 >R_0$. Then
\begin{equation}\label{eq:3.6}
\frac{|\nabla \omega |^2}{\omega^2} \leq \max\left\{ \Omega_1; \frac{4d\Omega_2 + \sqrt{(4d\Omega_2)^2 + 4\Omega_3}}{2}  \right\},
\end{equation}
on $B_{\gamma R_1} (o) \setminus \overline{B}_{R_1}(o)$, where
\begin{alignat*}{2}
&\Omega_1 :=&&\, \max\{ \omega^{-1}f_1(r)f_2(\omega) : x \in \overline{B}_{(\gamma +t)R_1} (o)\setminus B_{(1-t)R_1}(o) \};\\
&\Omega_2 :=&&\, \frac{A_1}{R_1}\left( \frac{1}{R_1} + 4(d-1)\max\left\{\sqrt{\bar{G}}; \frac{1}{R_1}  \right\}\right)  + \frac{(2+4d)A_1}{R_1^{2}}+ 2 (d-1)\bar{G}\\
& &&+ \max\{ 2f_1(r)\max\{(\omega^{-1}f_2(\omega) - f_2'(\omega));0\} +2\omega^{-1}L|f_1'(r)|^{2\lambda}|f_2(\omega)|: x \in \textnormal{\textbf{D}}_{\gamma,t,R_1}(o)\};\\
&\Omega_3 :=&& \max\left\{\omega^{-1}L|f_1'(r)|^{2(1-\lambda)}|f_2(\omega)|: x \in \textnormal{\textbf{D}}_{\gamma,t,R_1}(o)\right\},
\end{alignat*}
and
\begin{equation*}
\textnormal{\textbf{D}}_{\gamma,t,R_1}(o):=\overline{B}_{(\gamma +t)R_1} (o)\setminus B_{(1-t)R_1}(o),\quad A_1=A_1(t),\quad
\bar{G}:= \max\{ G(r):r \in [(1-t)R_1, (\gamma +t)R_1] \}.
\end{equation*}
The parameter $\lambda>0$ can be chosen in such a way as to minimize the right hand side of \eqref{eq:3.6}.
\end{theorem}
\begin{proof}
We adapt some of the ideas in the proof of \cite[Theorem 7.1]{cheeger2001degeneration}.
Let  $t >0$ be as in the statement, fix $x_i \in \partial B_{\frac{\gamma + 1}{2}R_1}(o)$ and consider the ball $B_{\left( \frac{\gamma-1}{2} + t\right)R_1}(x_i)$. Since $B_{\left( \frac{\gamma-1}{2} + t\right)R_1}(x_i) \subset B_{(\gamma + t)R_1}(o) \setminus \overline{B}_{(1-t)R_1}(o) \subset M\setminus \overline{B}_{R_0}(o)$,  $\omega$ satisfies \eqref{eq:3.14} on $B_{\left( \frac{\gamma-1}{2} + t\right)R_1}(x_i)$, so that, defining $v = \log \omega$, we have
\begin{equation}\label{eq:3.9}
| \nabla v| = \frac{| \nabla \omega|}{\omega}, \qquad \Delta v = - |\nabla v|^2 + f_1(\zeta)F_2(v),
\end{equation}
where $F_2(v) = e^{-v} f_2(e^v) = \omega^{-1}f_2(\omega)$.
Set now
$$
Q = \vartheta | \nabla v|^2,
$$
where the radial function $\vartheta : B_{\frac{\gamma -1}{2}R}(x_i) \to [0,1]$ satisfies
\begin{align}
&\vartheta (y) = \psi (s_i(y)) \quad \textnormal{ with } \psi \in C^\infty([0,+\infty)), \quad s_i(y) = \textnormal{dist}_M(y, x_i)\\
&\psi_{|[0, \frac{\gamma -1}{2}R_1]}(s_i) \equiv 1, \\
&\textnormal{supp }\psi \subset \left[0, \left(\frac{\gamma -1}{2} +t \right)R_1\right),\\
&-\frac{A_1(t)}{R_1} \sqrt{\psi} \leq \psi' \leq 0 \quad \textnormal{on } \left[\frac{\gamma -1}{2}R_1, \left(\frac{\gamma -1}{2}+t\right)R_1\right),\label{eq:3.13.4}\\
&|\psi''|\leq \frac{A_1(t)}{R_1 ^2} \quad \textnormal{on } \left[\frac{\gamma -1}{2}R_1, \left(\frac{\gamma -1}{2}+t\right)R_1\right), \label{eq:3.13.5}
\end{align}
and then
\begin{align*}
&\vartheta_{|\overline{B}_{\frac{\gamma -1}{2}R_1}(x_i)} \equiv 1, \\
&\textnormal{supp }\vartheta \subset B_{ \left(\frac{\gamma -1}{2}+t\right)R_1}(x_i).\\
\end{align*}
The function $Q$ takes on its maximum at some point $q_i \in B_{ \left(\frac{\gamma -1}{2}+t\right)R_1}(x_i)$. For now, consider $q_i$ not to be a cut point of $x_i$. Therefore, at $q_i$ we have $\nabla Q=0$ and $\Delta Q\leq 0$. Thus, at $q_i$,
$$
\nabla |\nabla v|^2 = - \vartheta^{-2} Q \nabla\vartheta.
$$
and
\begin{align}
\Delta Q &= \Delta \vartheta |\nabla v|^2 + 2 \langle \nabla \vartheta, \nabla | \nabla v |^2 \rangle + \vartheta \Delta |\nabla v|^2 \nonumber\\
&= (\vartheta^{-1} \Delta \vartheta - 2 \vartheta^{-2} |\nabla \vartheta |^2)Q + \vartheta \Delta |\nabla v|^2 \nonumber\\
&= (\vartheta^{-1} \Delta \vartheta - 2 \vartheta^{-2} |\nabla \vartheta |^2)Q + 2 \vartheta \left( |\text{Hess} (v)|^2 + \langle \nabla \Delta v, \nabla v \rangle + \text{Ric}_M(\nabla v, \nabla v) \right) \nonumber\\
&\geq (\vartheta^{-1} \Delta \vartheta - 2 \vartheta^{-2} |\nabla \vartheta |^2)Q + 2 \vartheta \left( |\text{Hess} (v)|^2 + \langle \nabla \Delta v, \nabla v \rangle - (d-1)G(r) |\nabla v|^2 \right) \label{eq:thm2.1}
\end{align}
where in the last equality we used the Bochner's formula.
Note that
\begin{align}
2 \vartheta |\text{Hess} (v)|^2 &\geq \frac{2\vartheta}{d} (\Delta v)^2\nonumber\\
&= \frac{2}{d} \vartheta^{-1} (-Q + \vartheta f_1(\zeta)F_2(v))^2.\label{eq:them2.3}
\end{align}
Moreover, for any $\alpha,\beta >0$,
\begin{align*}
2 \vartheta \langle \nabla \Delta v, \nabla v \rangle &= 2 \vartheta \langle \nabla(-|\nabla v|^2 + f_1(\zeta)F_2(v)), \nabla v\rangle\\
&=2\vartheta \langle \nabla (f_1(\zeta)F_2(v)), \nabla v\rangle - 2\vartheta \langle \nabla |\nabla v|^2, \nabla v\rangle\\
&= 2 f_1(\zeta) F_2'(v) Q + 2 \vartheta^{-1}\langle \nabla\vartheta, \nabla v \rangle Q + 2\vartheta f_1'(\zeta) F_2(v) \langle \nabla \zeta, \nabla v \rangle\\
&\geq 2 f_1(\zeta) F_2'(v) Q + 2 \vartheta^{-1}\langle \nabla\vartheta, \nabla v \rangle Q - 2\vartheta L|f_1'(\zeta) F_2(v)||\nabla v| \\
&= 2 f_1(\zeta) F_2'(v) Q + 2 \vartheta^{-1}\langle \nabla\vartheta, \nabla v \rangle Q\\
& \;\;\;\;- 2\vartheta \sqrt{L}|f_1'(\zeta)|^{(1-\lambda)}|F_2(v)|^{1/2}\left(\sqrt{L}|f_1'(\zeta)|^{\lambda} |F_2(v)|^{1/2}|\nabla v|\right) \\
&\geq 2f_1(\zeta) F_2'(v) Q - \epsilon^{-1} \vartheta^{-2} |\nabla \vartheta |^2 Q - \epsilon \vartheta^{-1} Q^2  -\vartheta L|f_1'(\zeta)|^{2(1-\lambda)}|F_2(v)| \\
&\;\;\;\;- \vartheta L|f_1'(\zeta)|^{2\lambda}|F_2(v)||\nabla v|^2,
\end{align*}
whence, taking $\epsilon = \frac{1}{4d}$,
\begin{align}\label{eq:thm2.2}
2 \vartheta \langle \nabla \Delta v, \nabla v \rangle &\geq 2 \left( f_1(\zeta)F_2'(v) - L|f_1'(\zeta)|^{2\lambda}| F_2(v)|\right)Q - 4d \vartheta^{-2} |\nabla \vartheta |^2 Q - \frac{Q^2}{4d} \vartheta ^{-1} \nonumber\\
& \;\;\;\; - L|f_1'(\zeta)|^{2(1-\lambda)}|F_2(v)|\frac{Q}{2}.
\end{align}
Inserting \eqref{eq:them2.3} and \eqref{eq:thm2.2} into \eqref{eq:thm2.1} and multiplying by $\vartheta$ yield
\begin{align}\label{eq:3.7}
\frac{2}{d} (-Q + \vartheta f_1 F_2)^2 - \frac{Q^2}{4d} &\leq \left[- \Delta \vartheta + (2+4d)\vartheta^{-1}|\nabla \vartheta|^2 + 2 (d-1)G(r) \vartheta \right. \nonumber\\
 &\left. \;\;\;\; - 2(f_1F_2' - L|f_1'|^{2\lambda}|F_2|) \vartheta\right]Q  + \vartheta^2L|f_1'|^{2(1-\lambda)}|F_2 |.
\end{align}
If $Q \leq 2\vartheta f_1 F_2$, then $|\nabla v|^2 \leq 2 f_1 F_2 = 2\omega^{-1}f_1(\zeta)f_2(\omega)$ and \eqref{eq:3.6} holds. If not,
$$
-Q + \vartheta f_1 F_2 \leq - Q/2 \leq 0,
$$
and
$$
\frac{2}{d} (-Q + \vartheta f_1 F_2)^2 - \frac{Q^2}{4d}  \geq \frac{Q^2}{4d}.
$$
In this case, using \eqref{eq:3.7} and the fact that $r(y) \in ((1-t)R_1, (\gamma +t) R_1)$,  and setting \\ $\bar{G}:= \max \{ G(r): r \in [(1-t)R_1, (\gamma +t) R_1]  \}$ we get
\begin{align}
Q^2 &\leq 4d \left[ - \Delta \vartheta + (2+4d)\vartheta^{-1}|\nabla \vartheta|^2 +2(d-1)G(r(y))\vartheta - 2(f_1F_2' - L|f_1'|^{2\lambda}|F_2|) \vartheta    \right]Q  \nonumber\\
& \;\;\;\;+ \vartheta^2L|f_1'|^{2(1-\lambda)}|F_2 |\nonumber\\
&\leq 4d \left[ - \Delta \vartheta + (2+4d)\vartheta^{-1}|\nabla \vartheta|^2 + 2(d-1)\bar{G} \vartheta +  2(f_1F_2' - L|f_1'|^{2\lambda}|F_2|) \vartheta    \right]Q  \nonumber\\
& \;\;\;\;+ \vartheta^2L|f_1'|^{2(1-\lambda)}|F_2 | \nonumber\\
&= 4d \left[ A_2(d,\kappa, \alpha, t) + 2 (d-1)\bar{G}\vartheta + 2f_1(\zeta)(\omega^{-1}f_2(\omega) - f_2'(\omega)) \vartheta +2\omega^{-1}L|f_1'(\zeta)|^{2\lambda}|f_2(\omega)|\vartheta   \right]Q\nonumber\\
& \;\;\;\;+ \vartheta^2L|f_1'|^{2(1-\lambda)}|F_2 | \nonumber\\
&\leq  4d \left[ A_2(d,\kappa, \alpha, t) + 2 (d-1)\bar{G} + 2f_1(\zeta)\max\left\{\omega^{-1}f_2(\omega) - f_2'(\omega);0\right\} +2\omega^{-1}L|f_1'(\zeta)|^{2\lambda}|f_2(\omega)|  \right]Q \nonumber\\
& \;\;\;\;  +L|f_1'|^{2(1-\lambda)}|F_2 |  \label{eq:3.8}
\end{align}
where
\begin{equation}\label{eq:3.12.4}
 A_2(d, \kappa, \alpha, t) = - \Delta \vartheta + (2+4d)\vartheta^{-1}|\nabla \vartheta|^2 \leq  - \Delta \vartheta + (2+4d)A_1R_1^{-2},
\end{equation}
by \eqref{eq:3.13.4}. Thus, we have
$$
0\leq Q \leq \frac{4d\tilde{\Omega}_2 + \sqrt{(4d\tilde{\Omega}_2)^2 + 4\tilde{\Omega}_3}}{2},
$$
with
\begin{align*}
&\tilde{\Omega}_2=   A_2(d,\kappa, \alpha, t) + 2 (d-1)\bar{G} + 2f_1(\zeta)\max\left\{\omega^{-1}f_2(\omega) - f_2'(\omega);0\right\} +2\omega^{-1}L|f_1'(\zeta)|^{2\lambda}|f_2(\omega)|  , \\
& \tilde{\Omega}_3= L|f_1'|^{2(1-\lambda)}|F_2 |.
\end{align*}
To conclude it remains to show that $A_2$ is bounded.  and  \eqref{eq:3.6} will follow. Indeed,
$$
\Delta \vartheta = \psi''(s_i) + \psi'(s_i)\Delta s_i,
$$
is not identically zero only for $s_i \in \left(\frac{\gamma -1}{2}R_1, (\frac{\gamma -1}{2} + t)R_1\right)$ and since  for every \\$y \in B_{\left(\frac{\gamma-1}{2}+t\right)R_1}(x_i)
$,
\begin{equation*}\label{eq:3.16}
\textnormal{Ric}_M (\nabla s_i(y), \nabla s_i (y))  \geq - (d-1)G(r(y))\geq -(d-1) \bar{G},
\end{equation*}
using Laplacian comparison, $\psi'\leq0$, \eqref{eq:3.13.4} and \eqref{eq:3.13.5}, we deduce that
\begin{align*}\label{eq:3.11}
\Delta \vartheta &\geq  \psi''(s_i) + (d-1)\sqrt{\bar{G}}\coth\left(\sqrt{\bar{G}}s_i \right) \psi'(s_i)\\
&\geq \psi''(s_i) + \max \left\{ 2(d-1)\sqrt{\bar{G}}; \frac{4(d-1)}{(\gamma-1)R_1}   \right\}\psi'(s_i)\\
&\geq -\frac{A_1}{R_1^2} - 4(d-1)\max\left\{\sqrt{\bar{G}}; \frac{1}{R_1}  \right\} \frac{\sqrt{\psi}A_1}{R_1}\\
&\geq - \frac{A_1}{R_1}\left( \frac{1}{R_1} + 4(d-1)\max\left\{\sqrt{\bar{G}}; \frac{1}{R_1}  \right\} \right).
\end{align*}
The above inequality holds pointwise whenever $q_i$ is not a cut point of $x_i$. If $q_i$ is a cut point, in order to have $\vartheta$ smooth in a neighborhood of $q_i$, we can use a standard argument by Calabi, replacing $s_i(y)$ with its associated upper barrier function $s_{i,\epsilon, q_i}(y)$ in the definition of $\vartheta$, i.e., $\vartheta (y) = \psi(s_{i,\epsilon, q_i}(y))$, where $s_{i,\epsilon, q_i}(y) = \epsilon + \textnormal{dist}_M(\delta(\epsilon),y) = \epsilon + r_{\delta(\epsilon)}(y)$ and $\delta$ is the minimum geodesic joining $x_i$ to $q_i$. Since $\psi$ is nonincreasing, then $q_i$ is still a maximum for $Q$ and the above estimates hold again. Hence, we proved that on $B_{\left(\frac{\gamma -1 }{2}+t\right)R_1}(x_i)$
\begin{equation}\label{eq:3.15}
\vartheta\frac{|\nabla \omega (x) |^2}{\omega^2 (x)} \leq \max\left\{ \Omega_{1,i}; \frac{4d\Omega_2 + \sqrt{(4d\Omega_{2,i})^2 + 4\Omega_{3,i}}}{2}  \right\},
\end{equation}
where
\begin{alignat*}{2}
&\Omega_{1,i} :=&& \max\{ \omega^{-1}f_1(\zeta)f_2(\omega) : x \in \overline{B}_{\left(\frac{\gamma -1 }{2}+t\right)R_1}(x_i) \};\\
&\Omega_{2,i}:=&& A_3(d,\kappa,\gamma,t,\bar{G},R_1)+  \max\{ 2f_1(\zeta)\max\{\omega^{-1}f_2(\omega)- f_2'(\omega);0 \} \\
&  && +2\omega^{-1}L|f_1'(\zeta)|^{2\lambda}|f_2(\omega)|: x \in \overline{B}_{\left(\frac{\gamma -1 }{2}+t\right)R_1}(x_i) \},\\
&\Omega_{3,i} :=&&  \max\{\omega^{-1}L|f_1'(\zeta)|^{2(1-\lambda)}|f_2(\omega) | : x \in \overline{B}_{\left(\frac{\gamma -1 }{2}+t\right)R_1}(x_i)    \},
\end{alignat*}
and
\begin{align*}
A_3(d,\kappa,\gamma,t,\bar{G}, R_1):=& \frac{A_1}{R_1}\left( \frac{1}{R_1} + 4(d-1)\max\left\{\sqrt{\bar{G}}; \frac{1}{R_1}  \right\}\right)\\
&+ \frac{(2+4d)A_1}{R_1^{2}}+ 2 (d-1)\bar{G}.
\end{align*}

Now, by compactness, there exists a finite collection $\{ x_i \}_{i=1}^n \subset \partial B_{\frac{\gamma + 1}{2}R}(o)$ such that
$$
\bigcup_{i=1}^n B_{\left(\frac{\gamma-1}{2}+t\right)R_1}(x_i) \supset \overline{B}_{\gamma R_1}(p) \setminus B_{R_1}(o).
$$
Then, choosing $\Omega_1 = \max\{\Omega_{1,i}\}, \Omega_2 = \max\{ \Omega_{2,i} \}$ and $\Omega_3 = \max\{ \Omega_{3,i} \}$, the thesis follows.
\end{proof}
\begin{remark}\label{rem:3}
The constant $A_1(t) \to \infty$ as $R_1 \to R_0$. Moreover, the above theorem can be extended easily to the case $\gamma = \infty$ if $\sup G(r) < \infty$ and to the case where $\omega$ is defined only on an annulus $B_{\gamma R}(o) \setminus \overline{B}_R(o)$, $R>1$, namely $\omega : B_{\gamma R}(o) \setminus \overline{B}_{R}(o)  \to \R$ such that
\begin{equation*}
\begin{cases}
\omega > 0 & \textnormal{on }   B_{\gamma R}(p)\setminus \overline{B}_{R}(o), \\
\Delta \omega = f_1(\zeta)f_2(\omega).
\end{cases}
\end{equation*}
In this latter case, the estimate \eqref{eq:3.6} still holds in any inner annulus of the form\\ $B_{(\gamma -\theta) R}(o) \setminus \overline{B}_{(1+\theta)R}(o)$, provided  $0<\theta< \frac{\gamma +1}{2}$, and replacing $\textnormal{\textbf{D}}_{\gamma,t,R_1}(o)$ with \\ $\textnormal{\textbf{D}}_{\gamma,\theta, R}(o):=\overline{B}_{(\gamma - \theta/2)R} (o)\setminus B_{(1+\theta/2)R}(o)$.  Note that in this case $\Omega_2\to \infty$ for $\theta\to 0$, since now the $A_1=a_1(\theta) \to \infty$ as $\theta \to 0$.
\end{remark}

\begin{corollary}\label{cor:3.2}
	Let $\omega$ as in the previous Theorem \ref{thm:3.2} and let $G(r) = \frac{\kappa^2}{(1+r^2)^{\alpha/2}}$ with $\alpha \in [-2,2]$. If
	\begin{itemize}
		\item [(i)] $\Delta \omega = \frac{\omega}{r^{\alpha}}$,
	\end{itemize}
	or if
	\begin{itemize}
		\item [(ii)] $\Delta \omega \equiv \frac{1}{R_1^{\alpha}}$ and $\omega \geq C>0$, with $C$ independent of $R_1$,
	\end{itemize}
	then
	$$
	\frac{\left|\nabla\omega\right|^2}{\omega^2} \leq	\frac{A(d,\kappa,\gamma,\alpha,t)}{R_1^{\alpha}}.
	$$
\end{corollary}
\begin{proof}
Fix $f_1(\zeta)= f_1(r)= \frac{1}{r^\alpha}$ and $f_2(\omega)=\omega$, and choose $\lambda =\frac{1}{3}$ and $\lambda=\frac{1}{2}$ for $\alpha \in [0,2]$ and for $\alpha \in [-2,0)$, respectively. Then it is just a matter of easy calculations to see that
\begin{align*}
&\Omega_1 \leq \frac{A_1}{R_1^{\alpha}},\\
&\Omega_2 \leq \frac{A_2}{R_1^{1+\alpha/2}} + \frac{A_3}{R_1^{2}} + \frac{A_4}{R_1^\alpha} + \frac{A_5}{R_1^{2\lambda(\alpha +1)}}\leq \frac{A_6}{R_1^\alpha},\\
&\Omega_3 \leq \frac{A_7}{R_1^{2(1-\lambda)(\alpha +1)}},
\end{align*}
from which it follows that
$$
\frac{4d\Omega_2 + \sqrt{(4d\Omega_2)^2 + 4\Omega_3}}{2} \leq \frac{A_8}{R_1^\alpha}.
$$
If $f_2(\omega)\equiv \frac{1}{R_1^{\alpha}}$ instead and $\omega$ is uniformly bounded from below by a constant $C$, then
$$
\omega^{-1}f_2(\omega) - f_2'(\omega)= \omega^{-1}f_2(\omega) \leq \frac{1}{CR_1^{\alpha}},
$$
and the thesis follows from the same estimates of above.
\end{proof}

We next prove a lower estimate for the volume of ball of a fixed (small) radius in terms of the distance of their center from a fixed point under radial bounds on the Ricci curvature. It generalizes similar estimates known when the Ricci curvature is bounded below by a constant. Note that having a variable lower bound on Ricci makes the geometry no longer homogeneous and therefore requires a significantly more careful analysis.
\begin{proposition}\label{prop:3.4}
Suppose that
\[
Ric \geq (d-1)\frac{\kappa^2}{(1+r(x)^2)^{\alpha/2}}, \quad \alpha\in [-2,2].
\]
Then, for every $x \in M \setminus \overline{B_{1} (o)}$, we have
$$
vol(B_{1/4}(x)) \geq \begin{cases}
\bar{C}_1 e^{-\bar{C}_2 r^{1-\alpha/2}(x)}, & \textnormal{for } \alpha \in [-2,2),\\
\bar{C}_1 r^{-[1+ (d-1)(1+\sqrt{1+4\kappa^2})]}, &  \textnormal{for } \alpha =2.
\end{cases}
$$
\end{proposition}
\begin{proof}
We will give a direct proof for $\alpha \in [0,2)$ while the case for $\alpha=2$ can be recovered by small modifications of the following considerations.

Let $x$ be fixed and define $s(y):= \textnormal{dist}_M(y,x)$. Then, by hypothesis it holds that
$$
Ric_M(\nabla s (y), \nabla s(y)) \geq -(d-1) \frac{\kappa^2}{\left(1 + r(y)^2 \right)^{\alpha/2}} \geq  -(d-1) \frac{\kappa^2}{\left(1 + |r(x) -s(y)|^2 \right)^{\alpha/2}},
$$
namely
$$
Ric_M(\nabla s (y), \nabla s(y)) \geq - (d-1) G(s),
$$
with $G(s) = \kappa^2/\left(1 + |r(x) -s(y)|^2 \right)^{\alpha/2}$. Let $h(s) \in C^2([0, r(x)])$ be the solution of the problem
\begin{equation}\label{eq:prop7.1}
\begin{cases}
h''(s) = G(s) y(s),\\
h(0)=0,\\
h'(0)=1,
\end{cases}
\end{equation}
on $[0,r(x)]$, and let $\psi(s) \in C^2([0, r(x)])$ be the solution of the problem
$$
\begin{cases}
\psi''(s) = \frac{\kappa^2}{(r(x)- s)^\alpha}\psi(s),\\
\psi(0)=0,\\
\psi'(0)=1,
\end{cases}
$$
on $[0,r(x)]$. The existence of $\psi$ follows from Lemma \ref{lem:3.1}, and, since $\kappa^2 /(r(x)- s)^\alpha \geq G(s)$, we can apply Lemma \ref{lem:3.4} te get
$$
0 \leq h(s) \leq \psi (s) \qquad \textnormal{on } [0, r(x)].
$$
Since $r(x) \geq 1$, by Corollary \ref{cor:1.2} we have that
\begin{equation}\label{eq:3.17}
\frac{vol\left( B_{r(x)}(x) \right)}{V_G (r(x))} \leq \frac{vol\left( B_{1/4}(x) \right)}{V_G (1/4)} \leq \hat{C}_1 vol\left( B_{1/4}(x) \right).
\end{equation}

Now, let $\beta_p(t)$ be a minimizing geodesic parametrized by arc length connecting $x$ to $o$ and fix $\bar{o} = \beta(r(x)-1)$. Then, $\bar{o} \in S_1(o)$ and for every $y \in B_1 (\bar{o})$ it holds that
$$
\textnormal{dist}_M(y,x) \leq \textnormal{dist}_M(y,\bar{o}) + \textnormal{dist}_M(x, \bar{o}) \leq r(x),
$$
namely, $B_{r(x)}(x)\supset B_1 (\bar{o})$. Since
$$
\min_{q \in S_1(p)} vol (B_1 (q)) \geq \hat{C}_2 >0,
$$
we have that
\begin{equation}\label{eq:3.18}
\frac{vol\left( B_{r(x)}(x) \right)}{V_G (r(x))} \geq \frac{ vol(B_{1} (\bar{o})) }{\hat{C}_3 \int_0 ^{r(x)} \psi(t)^{d-1}dt} \geq \frac{\hat{C}_2}{\hat{C}_3 r(x)^{1+ \frac{(d-1)\alpha}{4}} e^{\hat{C}_4 r^{1-\frac{\alpha}{2}}(x)}} \geq \frac{\hat{C}_2}{\hat{C}_3 e^{\hat{C}_5 r^{1-\frac{\alpha}{2}}(x)}} ,
\end{equation}
where the right hand side inequality comes from Lemma \ref{lem:3.1} and the previous observation. Combining  \eqref{eq:3.17} and \eqref{eq:3.18} we obtain the required concludion.
\end{proof}

\begin{lemma}\label{lem:3.1}
Let consider the following ODE problem on $[0, r)$, $\alpha \in [-2,2]$,
\begin{equation}\label{eq:3.19}
\begin{cases}
\psi''(s) = G(s)\psi(s),\\
\psi(0)=0,\\
\psi'(0)=1,
\end{cases}
\end{equation}
with $G(s) =\frac{\kappa^2}{(r - s)^\alpha}$. Then there exists an unique solution $\psi \in C^2([0,r))$ such that $\psi' >0$ on $[0,r)$ and
\begin{itemize}
	\item[(i)] \textbf{Case $\alpha \in [-2,0)$}
	\begin{align}
	\psi(s) &\leq C_1(r) \frac{2^{\alpha/2}}{\kappa}\sinh\left(\frac{2\kappa}{2-\alpha}\left[(1+(r-s))^{1-\alpha/2} -1\right]  \right) \nonumber\\
	&+ C_2(r) \frac{2^{\alpha/2}}{\kappa}\cosh\left(\frac{2\kappa}{2-\alpha}\left[(1+(r-s))^{1-\alpha/2} -1\right]  \right). \label{eq:3.20.2}
	\end{align}
\item [(ii)]\textbf{Case $\alpha \in [0,2)$}
\begin{equation}\label{eq:3.20.1}
\psi (s) = C_1(r) \sqrt{r - s}\, I_{\frac{1}{2-\alpha}}\left(\frac{\kappa}{1- \frac{\alpha}{2}} (r-s)^{1- \frac{\alpha}{2}} \right) + C_2(r) \sqrt{r - s}\, K_{\frac{1}{2-\alpha}}\left(\frac{\kappa}{1- \frac{\alpha}{2}} (r-s)^{1- \frac{\alpha}{2}} \right),
\end{equation}
where $I_\nu (z), K_\nu (z)$ are the modified Bessel functions.
\item[(iii)] \textbf{Case $\alpha =2$}
\begin{equation}\label{eq:3.20.3}
\psi (s) = C_1 (r) \left(r-s\right)^{\frac{1+\sqrt{1+4\kappa^2}}{2}} + C_2(r) \left(r-s\right)^{\frac{1-\sqrt{1+4\kappa^2}}{2}}.
\end{equation}
\end{itemize}
 Moreover, for $r\geq 1$ it holds that
$$
\begin{cases}
\psi(r) \leq C_3 r^{\alpha/2}e^{C_4 r^{1-\alpha/2}}, & \alpha \in [-2, 0],\\
\psi(r) \leq C_3 r^{\alpha/4} e^{C_4 r^{1- \frac{\alpha}{2}}}, & \alpha \in [0,2),\\
\psi(r-1) \leq \frac{r^{1+ \sqrt{1+4\kappa^2}}}{\sqrt{1+4\kappa^2}}, & \alpha=2,
\end{cases}
$$
with $C_3$ and $C_4$ constants that depend only on $\alpha$ and $\kappa$.
\end{lemma}
\begin{proof}
	\begin{itemize}
		\item[(i)]\textbf{Case $\alpha \in [-2,0]$}.
	
		It is not difficult to prove that the right hand side of \eqref{eq:3.20.2} is a subsolution of \eqref{eq:3.19}. From the initial conditions we get that
		\begin{align*}
		&C_1(r)= -\left(\frac{1+r}{2} \right)^{\alpha/2}\cosh\left(\frac{2\kappa}{2-\alpha} \left[(1+r)^{1-\alpha/2}-1\right] \right),\\
		&C_2(r) = \left(\frac{1+r}{2} \right)^{\alpha/2}\sinh\left(\frac{2\kappa}{2-\alpha}\left[(1+r)^{1-\alpha/2}-1\right] \right),
		\end{align*}	
		and then for $r\geq 1$ it follows that
		\begin{align*}
		\psi(r) \leq C_3r^{\alpha/2}e^{C_4 r^{1-\alpha/2}}.
		\end{align*}
	
		\item [(ii)] \textbf{Case $\alpha \in [0,2)$}.

	By a change of variable $x= r-s$, it is easy to check (see \cite[pp. 374-379]{abramowitz1964handbook}) that a general solution of the problem \eqref{eq:3.19} can be expressed in the form of \eqref{eq:3.20.1}. Imposing $\psi(0)=0$ it gives
	$$
	C_1(r) = -C_2(r) \frac{K_{\frac{1}{2-\alpha}}\left(\frac{\kappa}{1- \frac{\alpha}{2}} r^{1- \frac{\alpha}{2}} \right)}{I_{\frac{1}{2-\alpha}}\left(\frac{\kappa}{1- \frac{\alpha}{2}} r^{1- \frac{\alpha}{2}} \right)}.
	$$
Using of the following properties
\begin{align*}
&\frac{d I_\nu (z)}{dz}(z) = \frac{1}{2}\left( I_{\nu +1} (z) + I_{\nu -1} (z) \right),\\
&\frac{d K_\nu (z)}{dz}(z) = \frac{1}{2}\left( K_{\nu +1} (z) + K_{\nu -1} (z) \right),
\end{align*}
and defining $z_r = \frac{\kappa}{1- \frac{\alpha}{2}} r^{1- \frac{\alpha}{2}}$, we get
\begin{align*}
\psi'(0)= &- C_1 \left\{ \frac{1}{2\sqrt{r}} I_{\frac{1}{2-\alpha}}\left(z_r \right) + \frac{\sqrt{r}}{2}\kappa r^{-\alpha/2} \left[  I_{\frac{1}{2-\alpha} +1}\left(z_r \right) + I_{\frac{1}{2-\alpha}-1}\left(z_r \right)\right]\right\}\\
&- C_2 \left\{ \frac{1}{2\sqrt{r}} K_{\frac{1}{2-\alpha}}\left(z_r \right) + \frac{\sqrt{r}}{2}\kappa r^{-\alpha/2} \left[  K_{\frac{1}{2-\alpha} +1}\left(z_r \right) + K_{\frac{1}{2-\alpha}-1}\left(z_r \right)\right]\right\},
\end{align*}
and since $\psi'(0)=1$,
$$
C_2(r) = \frac{1}{\frac{\kappa}{2} r^{\frac{1-\alpha}{2}} \left\{ K_{\frac{1}{2-\alpha}}\left(z_r \right) \left[ \frac{I_{\frac{1}{2-\alpha} +1}\left(z_r \right) + I_{\frac{1}{2-\alpha}-1}\left(z_r \right)}{I_{\frac{1}{2-\alpha}}\left(z_r \right)}  \right] - \left[K_{\frac{1}{2-\alpha} +1}\left(z_r \right) + K_{\frac{1}{2-\alpha}-1}\left( z_r \right)  \right] \right\}}.
$$
Making use of the fact that 
\begin{align*}
&I_\nu (0)=0,  &K_\nu(z) \sim \left(\frac{z}{2}\right)^\nu \Gamma(\nu+1) \quad\textnormal{for } z \to 0,\\
& I_\nu (z_r) \sim A_{\nu,1}  \frac{e^{z_r}}{\sqrt{2\pi z_r}}, &K_\nu (z_r) \sim A_{\nu,2} e^{-z_r}\sqrt{\frac{\pi}{2z_r}} \quad\textnormal{for large } z_r,\\ 
\end{align*}
we conclude that
$$
\psi(r) = C_2(r) C_5(\alpha) \leq C_3  r^{\alpha/4} e^{C_4 r^{1- \frac{\alpha}{2}}} \quad \textnormal{for every } r \geq 1,
$$
since $C_2(r)$ is of the same order at infinity of the right hand side.

\item[(iii)] \textbf{Case $\alpha=2$}.

It is just a matter of easy calculations to verify that $\psi$ satisfies \eqref{eq:3.20.3} with
$$
C_1(r)=
\frac
{
-r^{\frac{1 -\sqrt{1+4\kappa^2}} 2}
}
{\sqrt{1+4\kappa^2}}
,
\quad C_2(r)= \frac{r^{\frac{1+\sqrt{1+4\kappa^2} }{2}}}{\sqrt{1+4\kappa^2}}.
$$
\end{itemize}
	Finally, since $\psi''(s) \geq 0$ for every $s$ and $\psi'(0)=1$, then $\psi'>0$.
\end{proof}

The following Sturm-Liouville comparison result, which we state without proof, is at the basis of all comparison results valid under Ricci curvature lower bounds.
\begin{lemma}\label{lem:3.4}
	Let $G$ be a continuous function on $[0, r]$ and let $\phi, \psi \in C^1([0, \infty))$ with $\phi', \psi' \in \textnormal{AC}((0,\infty))$ be solutions of the problems
	$$
	\begin{cases}
	\phi'' -G\phi \leq 0 & \textnormal{a.e. in } (0,r),\\
	\phi(0)=0,
	\end{cases}
	\qquad
	\begin{cases}
	\psi'' -G\psi \geq 0 & \textnormal{a.e. in } (0,r),\\
	\psi(0)=0,\\
	\psi(0)>0.
	\end{cases}
	$$
	If $\phi(s)>0$ for $s\in (0,r)$ and $\psi'(0) \geq \phi'(0)$, then $\psi(s) >0$ in $(0,r)$ and
	\begin{itemize}
		\item[(i)] $\frac{\phi'}{\phi} \leq \frac{\psi'}{\psi}$,
		\item[(ii)] $\phi \leq \psi$.
	\end{itemize}
\end{lemma}
\begin{proof}
	See \cite[Lemma 2.1]{pigola2008vanishing}.
\end{proof}

\begin{corollary}\label{cor:1.2}
Assume that
\[
Ric\geq -(d-1) G(r(x))
\]
in the sense of quadratic forms with $G$ positive and $C^1$ on $[0,\infty)$ and let
$h$ be a solution of the differential inequality
$$
\begin{cases}
h'' -G h \geq 0 \\
h(0)=0,\\
h'(0)=1.
\end{cases}
$$
Then
\[
\Delta r\leq (d-1) \frac {h'(r(x)}{h(r(x))}
\]
pointwise in the complement of the cut-locus of $M$ and weakly on all of $M$. Moreover,
	for every $0\leq R_1\leq R_2$,
	\begin{equation}
	\frac{vol (B_{R_2})(o)}{V_G(R_2)} \leq \frac{vol (B_{R_1})(o)}{V_G(R_1)},
	\end{equation}
where $V_G(R)$ is the volume of the ball of radius $R$ centered at $o$ in the model manifold with radial Ricci curvature equal to $G$, namely,
\[
V_G(R)= c_d\int_0^R h(r)^{d-1}ds.
\]
\end{corollary}
\begin{proof}
	See \cite[Theorems 2.4 and  2.14]{pigola2008vanishing}.
\end{proof}

\begin{lemma}\label{lem:3.3}
	Set $\Omega = M \setminus (\{p \} \cup cut(o))$, and suppose that
	$$
	\Delta r(x) \leq \phi(r) \qquad \text{pointwise on } \Omega
	$$
	for some $\phi \in C^0 ([0, + \infty))$. Let $f \in C^2 (\R)$ be non-negative and set $F(x) = F(r(x))$ on $M$. Suppose either
	\begin{enumerate}[i)]
		\item $f' \leq 0,$ or
		\item $f' \geq 0.$
	\end{enumerate}
	Then, we respectively have
	\begin{enumerate}[i)]
		\item $\Delta F \geq f''(r) + \phi(r) f'(r);$
		\item $\Delta F \leq f''(r) + \phi(r) f'(r),$
	\end{enumerate}
	weakly on $M$.
\end{lemma}
\begin{proof}
See \cite[Lemma 2.5]{pigola2008vanishing}.
\end{proof}

\begin{proposition}\label{prop:3.1}
Let $Ric_M ( \nabla r , \nabla r) \geq  - (d-1) \frac{\kappa^2}{1+r^2}$, then
	$$
	\Delta r(x) \leq (d-1) C_\kappa  r^{-1} \qquad \textnormal{for every } r>0,
	$$
	in the sense of distributions on all of $M$, and with $C_\kappa = \frac{1+\sqrt{1+\kappa^2}}{2}$.
	\end{proposition}
\begin{proof}
	See \cite[Theorem 2.4 and Proposition 2.11]{pigola2008vanishing}.
\end{proof}

\begin{lemma}\label{lem:3.1.1}
	For every fixed $R \geq 1$ and for every $\gamma >1$, there exists a function $u : (0, +\infty) \to \R$ such that
	\begin{enumerate}[(i)]
		\item \label{eq:3.1.1} $u \in C^\infty ((0, +\infty))$ and $u '' (r) + \frac{a}{r}u '(r) = \frac{1}{\gamma^{a+1} r^2}$, where
		\item \label{eq:3.2.2} $u ' (r) < 0$ on $[R, \gamma R]$,
		\item $u(R)=1$ and $u(\gamma R)=0$.
	\end{enumerate}
\end{lemma}
\begin{proof}
A general solution of \eqref{eq:3.1.1} can be written in the form
		\begin{equation}\label{eq:lem3.1.1}
		u(r) = C_1 + C_2 r^{1-a} + \frac{r^2}{2\gamma ^{a+1} R^2(a+1)}.
		\end{equation}
		Since $u(R) = 1$, then
		\begin{equation}\label{eq:3.4}
		u(r) = C_2 (r^{1-a} - R^{1-a}) +1 + \frac{r^2 - R^2}{2\gamma^{a+1}R^2(a+1)}.
		\end{equation}
		In order to have $u'(r) <0$ on $[R, \gamma R]$, $C_2$ has to satisfy
		\begin{equation}\label{eq:3.3.2}
		C_2 > \frac{1}{a^2 -1}R^{a-1}.
		\end{equation}
		But condition  $u(\gamma R)=0$ is achieved if and only if
		\begin{equation}\label{eq:3.3}
		C_2 =  \frac{1 + \frac{\gamma^2 -1}{2(a+1)\gamma^{a+1}}}{(1- \gamma^{1-a})R^{1-a}},
		\end{equation}
		and putting together equations \eqref{eq:3.3.2} and \eqref{eq:3.3}, we get
		$$
		\frac{\gamma^{a-1}}{\gamma^{a-1}-1} + \frac{\gamma^2 -1}{2\gamma^2(a +1)(1-\gamma^{a-1})} > \frac{1}{a^2 -1},
		$$
		that is satisfied for every $R \geq 1$ and every $\gamma >1$. Hence, choosing $C_2$ as in \eqref{eq:3.3}, the thesis follows.
\end{proof}

\section{Applications. Gagliardo-Niremberg-type $\textnormal{L}^q$-estimates for the gradient and essential self-adjointness of Schroedinger-type operators.}\label{section:application1}
As previously mentioned, in \cite[Theorem 2.2]{guneysu2016sequences},  B. G\"{u}neysu established the existence of a sequence of Laplacian cut-off assuming that the Ricci curvature is  nonnegative, and then deduced a number of deep results using the cut-offs he constructed.  All the results in that paper which depend only on the existence of sequences of cut-off functions can be generalized to the geometric setting we consider. By way of example,  \cite[Theorem 2.3]{guneysu2016sequences} on $\Lnormal^q$ properties of the gradient, can be extended as follows.

Let us introduce the space
\begin{align*}
&\textnormal{L}^{2}_\alpha (M) := \left\{f:M \to \mathbb{R}\; : \; \int_M \frac{|f(x)|^2}{\left(1 + r^2(x)\right)^{\alpha/2}} dx  < \infty  \right\},\\
&\|f\|_{2,\alpha} := \left( \int_M \frac{|f(x)|^2}{\left(1 + r^2(x)\right)^{\alpha/2}} dx \right)^{1/2}.
\end{align*}
\begin{theorem}
Let $M$ be like in Theorem \ref{thm:3.1}, $\bar{\alpha}:= \min\{\alpha;0\}$ and let
$$
\textnormal{F}_{\bar{\alpha}}(M) := \left\{ f | f \in C^2(M)\cap \textnormal{L}^\infty (M) \cap \textnormal{L}^2(M), \; |\nabla f| \in \textnormal{L}^2_{\bar{\alpha}} (M), \; \Delta f \in \textnormal{L}^2(M)  \right\}.
$$
Then one has
$$
|\nabla f| \in \bigcap_{q \in [2,4]} \textnormal{L}^q(M) \quad \textnormal{for any }f \in \textnormal{F}_{\bar{\alpha}}(M).
$$
More precisely, for all of $f \in \textnormal{F}_{\bar{\alpha}} (M)$ one has
$$
\| \nabla f \|^2_2 = \langle f, -\Delta f \rangle, \qquad \|\nabla f \|^4_4 \leq (2+\sqrt{d})^2 \|f\|^2_\infty \left( \|\Delta\|_2^2 + (d-1)\kappa\|\nabla f \|^2_{2,\bar{\alpha}} \right).
$$
\end{theorem}
\begin{proof}
We give only a sketch of the proof since it can be adapted easily from the arguments presented in \cite{guneysu2016sequences}. We also remark that the condition $|\nabla f| \in \textnormal{L}^2_{\bar{\alpha}} (M)$ is necessary only for $\alpha \in [-2,0)$, since $\textnormal{L}^2(M) \subset \textnormal{L}^2_{\alpha} (M)$ for every $\alpha \in [0,2]$, and if $f \in \textnormal{L}^2(M)$ and $\Delta f \in \textnormal{L}^2(M)$ then $|\nabla f| \in \textnormal{L}^2(M)$, see \cite{strichartz1983analysis}, from which it can be derived either the global integration by part identity in the thesis's statements.

From \cite[Lemma 2]{grummt2012essential} we have the inequality
\begin{equation*}
\int_M |\nabla f |^4 \, dx \leq (2+\sqrt{d})^2 \| f\|_\infty ^2 \left(\int_M |\Delta f|^2 \, dx - \int_M \textnormal{Ric}_M (\nabla f, \nabla f) \, dx \right).
\end{equation*}
Inserting into the above inequality the Laplacian cut-offs $\{\phi_R\}$ of Corollary \ref{cor:3.1} and taking into account the Ricci lower bound, we get
\begin{equation*}
\int_M |\nabla (\phi_R f)|^4 \, dx \leq (2+\sqrt{d})^2 \| \phi_R f\|_\infty ^2 \left(\int_M |\Delta (\phi_R f)|^2 \, dx + (d-1)\kappa\int_M \frac{|\nabla (\phi_R f)|^2}{\left(1 + r^2\right)^{\alpha/2}} \, dx \right).
\end{equation*}
Properties 3. and 4. in the definition of the Laplacian cut-offs and by dominated convergence imply that \begin{equation*}
\lim_{R\to \infty} \int_M |\Delta(\phi_R f)|^2 \, dx = \int_M |\Delta f|^2 \, dx, \qquad \lim_{R\to \infty} \int_M \frac{|\nabla (\phi_R f)|^2}{\left(1 + r^2\right)^{\alpha/2}} \, dx = \int_M \frac{|\nabla  f|^2}{\left(1 + r^2\right)^{\alpha/2}} \, dx,
\end{equation*}
and the required conclusion follows.	
\end{proof}

In another direction, one can investigate the positivity preserving property of Schr\"odinger operators considered by M. Braverman, O. Milatovic and M. Shubin \cite[equation (B.4)]{braverman2002essential}, and recently addressed in \cite[Section 2.4]{guneysu2016sequences}, namely, assuming that $u \in \textnormal{L}^2(M)$ satisfies
\begin{equation}\label{eq:sec1.2}
\left(b - \Delta \right)u= \nu \geq 0 \qquad \textnormal{in } D'(M),
\end{equation}
with $b>0$ a positive real number,   can one conclude that $u\geq 0$ a.e.? Here the inequality $\nu \geq 0$ that $\langle \nu, \phi \rangle \geq 0$ for every $\phi \in C^\infty_c (M)$, and is equivalent to the fact that $\nu$ is a positive measure.
As shown in \cite{braverman2002essential}, there is a connection between the positivity preserving property of Schr\"odinger operators for  certain functional classes and the essential self-adjointness of the operator, in particular, the essential self-adjointness of $b-\Delta$ on $C^\infty_c(M)$ can be proved using the fact that the operator is positivity preserving for $L^2(M)$ functions. Since it is well know that $\Delta$ is essentially self-adjoint on $C^\infty_c(M)$ whenever $M$ is geodesically complete,  Braverman Milatovic and M. Shubin made the following conjecture, \cite[Conjecture P]{braverman2002essential},
\begin{conjecture}[Conjecture P]
Let $M$ be geodesically complete. Then
\[
u\in L^2(M) \text{ and } (b-\Delta u) = \nu\geq 0 \Rightarrow u\geq 0 \,\,a.e,
\]
\end{conjecture}
and proved that a sufficient condition for the above Conjecture to hold is that $M$ supports a sequence of cut-off functions. As mentioned in the introduction they were able to prove the existence of such cut-offs under the assumption of bounded geometry. It is proved in \cite[Section 2.4]{guneysu2016sequences} that this holds for manifolds with nonnegative Ricci curvature (indeed, it is shown that in that case the positivity preserving property actually holds for functions in $L^q$ for every $q\in [1,\infty]$). As a consequence of our results we are able to further enlarge the class of manifolds for which Conjecture $P$ holds.
\begin{proposition}
Let $M$ ba a complete Riemannian manifold such that
\[
\textnormal{Ric}_M ( \cdot , \cdot)\geq - (d-1)\frac{\kappa^2}{(1+r^2)^{\alpha/2}},
\]
for some $\alpha> -2$. Then Conjecture P holds on $M$.
\end{proposition}

\section{Applications. The Porous Medium Equation (PME) and the Fast Diffusion Equation (FDE) for the Cauchy problem on Riemannian manifolds.}\label{section:application2}
Hereafter we consider $M$ to be a geodesically complete manifold of dimension $d$ with
\begin{equation}\label{eq:Ric}
\text{Ric}_M ( \cdot , \cdot) \geq  - (d-1)\kappa^2 \frac{1}{(1+r^2)^{\alpha/2}} \langle \cdot , \cdot \rangle
\end{equation}
in the sense of quadratic forms and with respect to a fixed reference point $o \in M$, with $\kappa \geq 0$ and $\alpha \in [-2,2]$. Moreover, $\gamma$ will be a fixed positive real value such that $\gamma > \Gamma(\alpha, \kappa, d)$ as in Corollary \ref{cor:3.1}, and we will use the notation  $u^m := |u|^{m-1}u$.

The  Cauchy problem on $M$
\begin{equation}\label{eq:5.1}
\qquad \begin{cases}
\partial_t u(t,x) = \Delta u^m (t,x) & \mbox{for } x \in (0,+\infty)\times M \\
u(0,x) = u_0 (x) & \mbox{for } x \in M,
\end{cases}
\end{equation}
 which is called Porous Medium Equation (PME) when the exponent $m>1$ and Fast Diffusion Equation (FDE) when $0<m<1$, has been widely studied in the Euclidean setting (see \cite{vazquez2006smoothing} and \cite{vazquez2007porous} for detailed surveys), and, in recent years, several papers studied the properties of the solutions of those equations in the Riemannian setting, see for example \cite{bonforte2008fast}, \cite{grillo2014radial}, \cite{lu2009local}, \cite{vazquez2015fundamental}  and \cite{grillo2016smoothing}.

This Section is devoted to extensions and refinements of some results concerning solutions to  the PME and the FDE of the Cauchy problem is the setting of a Riemannian manifold satisfying  condition \eqref{eq:Ric}, mainly through the use of Laplacian cut-offs. 
The proofs that we propose here are often adaptations of the original proofs. For example, this is is the case, \cite[Proposition 9.1]{vazquez2007porous} compared to Proposition \ref{prop:sec2.1} and \cite[Lemma 3.1]{herrero1985cauchy} compared to Proposition \ref{prop:sec2.2},  but in order to make this paper reasonably self contained we will reproduce the more relevant details, whenever appropriate.

In Subsection \ref{subsection:1} we focus on the so called strong solutions of the PME proving $\Lnormal^1$-contractivity and conservation of mass properties. In Subsection \ref{subsection:2} we consider instead the FDE equation and generalize a weak-conservation of mass property, first proved in \cite{herrero1985cauchy} and then extended in \cite{bonforte2008fast} to the setting of  Cartan-Hadamard manifolds with bounded sectional curvature. We obtain an interesting lower bound on the extinction time $T(u_0)$ which depends explicitly on the lower bound on the Ricci curvature. In particular,  when \eqref{eq:Ric} holds with  $\alpha=2$ and $\kappa \geq 0$ in the Ricci inequality \eqref{eq:Ric} we get a generalization of the critical exponent $m_c$ (see \cite[Section 5]{vazquez2006smoothing}) below which finite time extinsion occurs, which reduces to the Euclidean value for $\kappa=0$, i.e., for  $\textnormal{Ric}\geq 0$. See Remark \ref{remark4} below.

It is worth to point out again that the only geometric assumption we make is geodesic completeness and the Ricci curvature lower bound \eqref{eq:Ric}. In particular we do not need hypotheses of topological nature nor to impose conditions on the  injectivity radius. In this sense, our results appear a genuine generalizations of previous results obtained on the PME/FDE-Cauchy problem posed in a Riemannian setting.
 
\subsection{$\textnormal{L}^1$ contractivity and uniqueness of the strong solution of the PME.}\label{subsection:1} Consider the Cauchy problem \eqref{eq:5.1} with $m>1$ and with initial datum $u_0$ which belongs to $\textnormal{L}^1 (M)$. 
\begin{definition}[Strong solutions for PME]\label{def:strong_PME}
Let $u \in C([0, \infty) : \textnormal{L}^1 (\R^d))$ be such that
\begin{enumerate}[(i)]
	\item \begin{equation}\label{PME:property1} u(0,x)=u_0; \end{equation}
	\item \begin{equation}\label{PME:property2}u^m \in \textnormal{L}_{\textnormal{loc}}^1 ((0, +\infty) : \textnormal{L}^1 (M)) \textnormal{ and } \partial_t u, \Delta u^m \in \textnormal{L}^1_{\textnormal{loc}}((0,+\infty)\times M);\end{equation}
	\item \begin{equation}\label{PME:property3}\partial_t u = \Delta (u^m) \textnormal{ a.e. in } (0,+\infty)\times M.\end{equation}
\end{enumerate}
Then $u$ is called strong solution for the Cauchy problem \eqref{eq:5.1} of the PME, see \cite[Definition 9.1]{vazquez2007porous}. In view of the next Proposition we will relax the request on $u^m$ in \eqref{PME:property2} asking only that
\begin{enumerate}
	\item[(ii')] $\partial_t u, \Delta u^m \in \textnormal{L}^1_{\textnormal{loc}}((0,+\infty)\times M)$ and \begin{equation}\label{PME:property4}\int_{t_1}^{t_2}\int_{\{x:\; n \leq r(x) \leq \gamma n \}} |u^m(t,x)| \, dtdx = o(n^{1+\alpha/2}) \, \textnormal{ as } n \to \infty, \tag{\ref*{PME:property2}'}\end{equation}
for every  $0<t_1<t_2$, with a fixed $\gamma>\Gamma$, see \cite[Remark p.197]{vazquez2007porous}.
\end{enumerate}
\end{definition}
In accordance with \cite[Proposition 9.1]{vazquez2007porous}, we have the following result.
\begin{proposition}\label{prop:sec2.1}
	Let $u$, $v$ be two strong solutions. For every $0<t_1<t_2$ we have
		\begin{equation}\label{eq:sec1.1}
		\int_M | u (t_2,x) - v (t_2,x)| dx \leq \int_M |u (t_1,x) - v (t_1,x)| dx.
		\end{equation}
\end{proposition}
\begin{proof}
	By (ii'), $\Delta u^m, \Delta v^m \in  \textnormal{L}^1_{\textnormal{loc}}((0,+\infty)\times M)$ and then it can be applied Kato's inequality \cite[Lemma A]{kato1972schrodinger}
	\begin{equation*}\label{eq:propSec2.1_1}
	- \Delta \left| u^m - v^m \right| \leq - \textnormal{sgn}(u-v) \Delta (u^m - v^m),
	\end{equation*}
	and by \eqref{PME:property3} we get
	\begin{equation*}
	\frac{d}{dt} |u- v|  \leq \Delta \left| u^m - v^m \right| \qquad \textnormal{in } D'((0,+\infty)\times M),
	\end{equation*}
	namely,
	\begin{equation*}
		\frac{d}{dt} \int_M \phi(x) |u(t) - v(t) | \, dx  \leq \int_M \Delta \phi(x) \left| u^m (t) - v^m(t) \right|\, dx
	\end{equation*}
	for every $\phi \in C^\infty_c(M)$. Then, integrating with respect to time and choosing $\phi=\phi_n$ a Laplacian cut-off functions as in Corollary \ref{cor:3.1}, we get
	\begin{align*}
\int_M \phi_n |u(t_2) - v(t_2) | \, dx  &\leq \int_M \phi_n |u(t_1) - v(t_1) | \, dx + \int_{t_1}^{t_2}\int_M \Delta \phi_n(x) \left| u^m (t) - v^m(t) \right|\, dx \\
&\leq  \int_M \phi_n |u(t_1) - v(t_1) | \, dx\\
&+ \|\Delta \phi_n(x) \|_{\infty} \int_{t_1}^{t_2}\int_{\{x\,:\, n\leq r(x)\leq \gamma n\}} \left| u^m (t) - v^m(t) \right|\, dx.
	\end{align*}
Letting $n \to \infty$, the required conclusion follows using  \eqref{PME:property4} and the estimate $\| \Delta_x \phi_n \|_\infty \leq C/n^{1+\alpha/2}$.	
\end{proof}

We have an immediate Corollary.
\begin{corollary}
Let $u, v$ be strong solutions of the Cauchy problem \ref{eq:5.1} with the same initial data, $u_0 = v_0$. Then $u = v$ almost everywhere. Moreover, the map $u_0 \mapsto u(t)$ is an ordered contraction in $\textnormal{L}^1 (M)$.
\end{corollary}

\begin{proposition}
For every $t>0$ we have
\begin{equation*}
\int_M u(t,x) = \int_M u_0.
\end{equation*}
\end{proposition}
\begin{proof}
We have that\begin{equation*}
\frac{d}{dt} \int_{M} \phi u(t) \,dx = \int_{M} \Delta\phi u^m \, dx,
\end{equation*}
in $D'(M)$ for every $\phi \in C^\infty_c(M)$. Then, taking Laplacian cut-offs $\phi=\phi_R$ and integrating in time the above equation in $[0,t]$, we get
\begin{align*}
\int_{M} \phi_R u(t) - \int_{M} \phi_R u(0) \, dx &= \int_0^t\int_M \Delta\phi_R u^m(s)\, dxdt\\
&\leq \|\Delta\phi_R\|_\infty \int_0^t\int_M |u^m(s)|\, dxdt\\
&\leq \frac{tC}{R^{1+\alpha/2}}\|u^m\|_1.
\end{align*}
We conclude letting $R$ going to infinity.
\end{proof}

\subsection{Weak conservation of mass of the FDE} \label{subsection:2}
Consider the Cauchy problem \eqref{eq:5.1} with $0<m<1$ and with initial datum $u_0$ in $\textnormal{L}^1_{\textnormal{loc}} (M)$.
\begin{definition}[weak and strong solutions for the FDE]
Let $u(t,x)  \in C([0, +\infty) :  \textnormal{L}^1_{\textnormal{loc}}(M) )$ be such that
\begin{enumerate}[(i)]
	\item \begin{equation}\label{FDE:property1}
u(0,x)=u_0,
\end{equation}
\item \begin{equation}\label{FDE:property2}
\partial_t u = \Delta u^m, \qquad \mbox{in } D' ((0, +\infty) \times M).
\end{equation}
\end{enumerate}
Then $u$ is called a weak solution for the Cauchy problem of the FDE. If moreover $u$ satisfies
\begin{enumerate}
\item[(iii)] \begin{equation}\label{FDE:property3}
\partial_t u \in \textnormal{L}^1_{\textnormal{loc}}((0,+\infty)\times M),
\end{equation}
\end{enumerate}
then $u$ is called a strong solution (see, \cite{herrero1985cauchy}). Notice that since $0<m<1$ then $u^m \in \textnormal{L}^1_{\textnormal{loc}}(M)$ as well.
\end{definition}

From \cite[Lemma 3.1]{herrero1985cauchy} we have the following Proposition.
\begin{proposition}\label{prop:sec2.2}
Let $u(t,x), v(t,x) \in \textnormal{L}^1_{\textnormal{loc}} (M)$. If $u(t,x)\geq v(t,x)$ are weak solutions of \eqref{eq:5.1} for the FDE, then for every $R \geq 1$ if $\alpha \in [-2,2)$, $R>0$ if $\alpha =2$, and for every  $\gamma >\Gamma_\alpha\geq1$, it holds
\begin{equation}\label{eq:6.1}
\left[  \int_{B_R(o)} \big(u(t_2,x) - v(t_2,x)\big) \, dx  \right]^{1-m} \leq \left[ \int_{B_{\gamma R(o)}}\big(u(t_1,x)- v(t_1,x)\big) \, dx   \right]^{1-m} + \mathcal{M}_{R, \gamma} (t_2-t_1),
\end{equation}
for every $0\leq t_1 \leq t_2$, where
\begin{equation}\label{constant:propsec2.2}
\mathcal{M}_{R, \gamma} = \frac{C}{R^{1+ \alpha/2}} \mbox{Vol}(B_{\gamma R}(o)\setminus B_{R}(o))^{1-m} >0,
\end{equation}
and where the constant $C$ is independent of $u$ and $v$ but depends only on $m, d, \kappa$ and $\gamma$.

If $u(t,x), v(t,x)$ are strong solutions of \eqref{eq:5.1} for the FDE, then it holds
\begin{equation}\label{eq:6.2}
\left[  \int_{B_R(o)} \big|u(t_2,x) - v(t_2,x)\big| \, dx  \right]^{1-m} \leq \left[ \int_{B_{\gamma R(o)}} \big|u(t_1,x)- v(t_1,x)\big| \, dx   \right]^{1-m} + \mathcal{M}_{R, \gamma} (t_2-t_1), \tag{\ref*{eq:6.1}'}
\end{equation}
where $\mathcal{M}_{R, \gamma}$ is exactly again \eqref{constant:propsec2.2}.
\end{proposition}
\begin{proof}
In the following, the constant $C$ can change from line to line and let us focus now on the first case, namely $u(t,x) \geq v(t,x)$ being weak solutions.

From \eqref{FDE:property2}, for every nonnegative $\eta \in C_c ^\infty (0,\infty)$ and $\psi \in C_c ^\infty (M)$ we have that
$$
\begin{matrix}
\langle \partial_t (u-v), \eta \psi \rangle & = & - \langle u - v, \partial_t \eta \psi \rangle \\
   \shortparallel & & \\
\langle \Delta (u^m - v^m), \eta \psi\rangle & = & \langle u^m - v^m, \eta \Delta \psi \rangle
\end{matrix}
$$
in distributions, that is,
\begin{equation*}
- \int_0 ^\infty \int_{M} \partial_t \eta \psi (u-v) \, dtdx = \int_0 ^\infty \int_{m} \eta \Delta \psi (u^m - v^m) \, dtdx,
\end{equation*}
namely
\begin{equation*}
- \int_0 ^\infty  \partial_t\eta \left(\int_{M} \psi (u-v) \,dx\right) \, dt = \int_0 ^\infty \eta \left( \int_{M}  \Delta \psi (u^m - v^m) \,dx\right) \, dt
\end{equation*}
and which implies
\begin{equation}\label{eq:5.2}
\frac{d}{dt} \int_{M} \psi (u(t) - v(t)) \,dx = \int_{M} \Delta \psi (u^m - v^m) \, dx
\end{equation}
in $D' (0,\infty)$ and in $\textnormal{L}^1_{\textnormal{loc}} (0,\infty)$ as well for every fixed $\psi$, as a consequence of \eqref{FDE:property1}. Since by concavity
$$
(r |r|^{m-1} - s |s|^{m-1}) \leq 2^{1-m} (r-s)^m \qquad \mbox{for all } r \geq s,
$$
then \eqref{eq:5.2} implies
$$
 \frac{d}{dt} \int_M \psi (u(t) - v(t)) \,dx  \leq 2^{1-m} \int_M |\Delta \psi | (u-v)^m.
$$
We set $g:= u-v$. By Holder's inequality, we obtain
\begin{equation}\label{eq:5.3}
 \frac{d}{dt} \int_M \psi g(t)    \leq C(\psi) \left[ \int_M \psi g(t)  \right]^m,
\end{equation}
where
$$
C(\psi) = \left[ 2 \int_{M} | \Delta \psi|^{1/(1-m)} \psi ^{-m/(1-m)}  \right]^{1-m}.
$$
Since the function $f_\psi (t) = \int_M \psi g(t)$ has weak derivative in $\textnormal{L}^1_{\textnormal{loc}}$, it is a.e. equal to an AC function, and by standard comparison arguments, for all  $ t_1,t_2 \geq 0 $ and every $\psi \in C_c ^\infty (M),$
\begin{equation}
 \left[\int \psi g(t_2)  \right]^{1-m} \leq \left[ \int \psi g(t_1)  \right]^{1-m} + (1-m) C(\psi) |t_2-t_1|.
\end{equation}
This will immediately imply the statement, once we prove that $C(\psi) \leq M_{R, \gamma} < \infty$.

Consider a function $\psi = \phi ^b \in C^2 _c (m)$, with $b > 2/(1-m)$ and $\phi$ as in Corollary \ref{cor:3.1}, namely $\phi : M \to [0,1]$ is such that
\begin{enumerate}[(i)]
\item $\phi_{|B_R (p)} \equiv 1$,
\item supp$(\phi) \subset B_{\gamma R} (o)$,
\item $| \nabla \phi | \leq \frac{C}{R}$,
\item $| \Delta \phi | \leq \frac{C}{R^{1 +\alpha/2}}$,
\end{enumerate}
where $C=C(d,\kappa,\alpha)$ is independent of $R$.

We then have,
\begin{align}\label{eq:5.4}
|\Delta(\psi(x))|^{1/(1-m)} \psi (x) ^{-m/(1-m)} &= \phi(x)^{-bm/(1-m)}\left| b(b-1) \phi^{b-2}|\nabla \phi |^2 + b \phi^{b-1} \Delta \phi \right|^{1/(1-m)} \nonumber\\
&\leq [b (b-1)]^{1/(1-m)} \phi^{[(b-2)-bm]/(1-m)} \cdot \left| |\nabla \phi |^2 + |\Delta \phi | \right| ^{1/(1-m)} \nonumber\\
& \leq [b (b-1)]^{1/(1-m)} \phi^{[(b-2)-bm]/(1-m)} \cdot C R ^{-\frac{1 + \alpha/2}{ 1-m}}.
\end{align}
An integration over $B_{\gamma R}(o) \setminus B_R(o)$, which contains  the support of $|\nabla \phi|$ and $\Delta \phi$, gives
\begin{align*}
C(\psi)&= \left[ 2 \int_{B_{\gamma R}(p) \setminus B_R(o)} | \Delta \psi|^{1/(1-m)} \psi ^{-m/(1-m)}  \right]^{1-m}\\
&\leq \frac{C}{R^{1+\alpha/2}} (\mbox{Vol}(B_{\gamma R}(o) \setminus B_R(o))^{1-m}.
\end{align*}

Let now $u(t,x), v(t,x)$ be strong solutions instead. According to  \eqref{FDE:property3}, $\Delta u^m, \Delta v^m \in  \textnormal{L}^1_{\textnormal{loc}}((0,+\infty)\times M)$ so that we can apply Kato's inequality \cite[Lemma A] {kato1972schrodinger} to get
\begin{equation}\label{eq:propSec2.2_1}
- \Delta \left| u^m - v^m \right| \leq - \textnormal{sgn}(u-v) \Delta (u^m - v^m),
\end{equation}
and then, using \eqref{FDE:property2} and arguing as in \cite[Theorem 2.3]{herrero1985cauchy}
\begin{equation*}
\frac{d}{dt} |u- v|  \leq \Delta \left| u^m - v^m \right| \qquad \textnormal{in } D'((0,+\infty)\times M).
\end{equation*}
The conclusion follows from  the same arguments used in the previous steps, and, in particular, from equality \eqref{eq:5.2}.
\end{proof}

\begin{remark}\label{remark4}
	Let $T(u_0)$ be the extinction time of the solution $u(t,x)$ with initial condition $u_0(x)$, namely $u(t,x)\equiv 0$ for every $t\geq T(u_0)$, see \cite{vazquez2006smoothing}. Let $v(t,x)\equiv 0$ and $s=0$. Then, if $\alpha=2$ in \eqref{eq:Ric}, we have
	$$
	T(u_0) \geq \frac{R^{2}}{C( vol(B_{\gamma R}(o))^{1-m}}\left(\int_{B_R(o)} u_0(x) dx\right)^{1-m}.
	$$
	Now, from the Bishop-Gromov inequality \eqref{eq:3.2} and \eqref{eq:V_G} applied with $r_1=\gamma R$, $r_2=1$, we have
	\begin{equation*}
	vol(B_{\gamma R}(o))\leq C V_G(\gamma R) = C \int_0 ^{\gamma R} h^{d-1}(t) dt,
	\end{equation*}
	but since $\tilde{h}(t) = t^{\frac{1+\sqrt{1+4\kappa^2}}{2}}$ is solution of \eqref{eq:3.19} for $G(t)= \kappa^2/t^2\geq \kappa^2/(1+t^2)$, then by Lemma \ref{lem:3.1} and Lemma \ref{lem:3.4} we can deduce that $h(t) \leq \tilde{h}(t)$ and get
	\begin{equation*}
	T(u_0)\geq   \bar{C} \frac{R^{2}}{R^{\left[1+\left(\frac{1+\sqrt{1+4\kappa^2}}{2}\right)(d-1)\right](1-m)}},
	\end{equation*}
	whence, letting $R \to \infty$, we deduce that $T(u_0)=\infty$ if
	$$
	2- \left[1+\left(\frac{1+\sqrt{1+4\kappa^2}}{2}\right)(d-1)\right](1-m)>0,
	$$
	that is, rearranging, provided

	\begin{equation}\label{eq:m_c}
 m>m_c= 1- \frac{2}{\left[1+\left(\frac{1+\sqrt{1+4\kappa^2}}{2}\right)(d-1)\right]}.
	\end{equation}
Note that, if $\textnormal{Ric}\geq 0$, so that we can take $\kappa =0$,  we recover the Euclidean constant $m_c= \frac{d-2}{d}$. On the other hand,
 if $\alpha \in [-2,2)$,   $vol(B_R (o))$ may grow super-polynomially, and, in general  we can not deduce a non-extinction property.
	Observe that, as stated in \cite[section 3 - examples 3.1]{grillo2016smoothing}, in a model manifold with radial Ricci curvature $Ric(\nabla r, \nabla r) = -(d-1)\frac{\kappa^2}{(1+r^2(x))^{\alpha/2}}$, $\alpha \in (0,2)$, radial functions satisfy a Sobolev-Poincar\'{e} inequality of the form
	\begin{equation}\label{Sobolev-Poincaré}
	\|f \|_{2\sigma}\leq C \| \nabla f\|_2, \qquad \sigma \in [1, d/(d-2)],
	\end{equation}
	 which is a key ingredient for a proof of finite extinction time. According to \cite[Theorem 6.1]{bonforte2008fast}, radial strong solutions of the FDE in such model manifolds vanish in a finite time $T(u_0)$ for every $m \in (0,1)$, provided that $u_0 \in \Lnormal^q(M)$ with $q\geq d(1-m)/2$.
\end{remark}
From Proposition \ref{prop:sec2.2} and Remark \ref{remark4}, we get
\begin{theorem}
Let $u(t,x) \in \Lloc(M)$ be a weak solution of \eqref{eq:5.1} for the FDE, then for every $R\geq 2$ if $\alpha \in [-2,2)$, $R\geq 1$ if $\alpha =2$, and for every $\gamma>\Gamma\geq1$, it holds
\begin{equation*}
\int_{B_R(o)} u(t,x) \, dx \leq 2^{1/(1-m)} \left\{ \int_{B_{\gamma R}(o)} u (s,x) \, dx + \left(\mathcal{M}_{R,\gamma} |t-s|\right)^{1/(1-m)}  \right\},
\end{equation*}
for any $t,s\geq0$ and where $\mathcal{M}_{R,\gamma}$ is like in \eqref{constant:propsec2.2}. If there exists an extinction time $T(u_0)$, then it is lower bounded by
\begin{equation*}
T(u_0) \geq  \frac{R^{1+\alpha/2}}{C( vol(B_{\gamma R}\setminus B_{R}))^{1-m}}\left(\int_{B_R} u_0(x) dx\right)^{1-m}.
\end{equation*}
\end{theorem}

Finally, let us observe that inequality \eqref{eq:6.1} depends on chosen reference point  $o$.  Thus, in order to prove  uniqueness of strong solutions for every $m \in (0,1)$ with the method of
 \cite[Theorem 2.3]{herrero1985cauchy}, the first task is to get rid of that dependency. But this alone is not enough, since a key tool there is the Mean Value Theorem for subharmonic functions. Keeping this into consideration, we can prove the following result.
\begin{theorem}
Let $M$ be a geodesically complete manifold and let $u, v$ be strong solutions for the FDE problem \eqref{eq:5.1} with same initial data, $u_0=v_0$. If
\begin{enumerate}[(i)]
\item $\textnormal{Ric}_M(\cdot, \cdot)$ satisfies \eqref{eq:Ric} with $\alpha=2$, then $u\equiv v$ for every $m>m_c$, where $m_c$ is defined as in \eqref{eq:m_c};
\item $\textnormal{Ric}_M(\cdot, \cdot) \geq 0$, so that \eqref{eq:Ric} holds with $\kappa=0$, then $u\equiv v$ for every $m \in (0,1)$.
\end{enumerate}
\end{theorem}
\begin{proof}
From inequality \eqref{eq:6.2}, we have
\begin{align}\label{eq:propSec2.2_2}
\int_{B_R(p)} |u(t)-v(t)| \, dx &\leq C \left[ \int_{B_{\gamma R}(o)} |u(0) -v(0)|\, dx + \frac{vol(B_{\gamma R}(o))}{R^{\frac{2}{1-m}}}t^{\frac{1}{1-m}}  \right]\nonumber\\
&= C\frac{vol(B_{\gamma R}(o))}{R^{\frac{2}{1-m}}}t^{\frac{1}{1-m}},
\end{align}
and observe that the above inequality is valid for both the cases (i) and (ii). From Remark \ref{remark4}
$$
vol(B_{\gamma R}(o)) \leq C(o) R^{1+\left(\frac{1+\sqrt{1+4\kappa^2}}{2}\right)(d-1)},
$$
and letting $R\to \infty$ in \eqref{eq:propSec2.2_2}, the right hand side converges to $0$ provided $m > m_c$ and the thesis follows for case (i).

Let us now be in case (ii), namely $\kappa =0$. Then inequality \eqref{eq:propSec2.2_2} is true for every $p \in M$. Set
\begin{equation*}
f(t,x) = \int_0 ^t |u^m - v^m|(s,x)\, ds.
\end{equation*}
By integrating in time in \eqref{eq:propSec2.2_1} we get $|u(t) - v(t)| \leq \Delta f(t,x)$ in $D'(M)$ for every $t>0$. Therefore, $f$ is subharmonic and from \cite[Theorem 2.1]{li1984p} it holds that
\begin{equation}\label{eq:propSec2.2_3}
f(t,p) \leq C vol(B_{R}(p))^{-1}\int_ {B_{R}(o)}f(t,x) \, dx,
\end{equation}
for every $R>0$ and for every $p \in M$, with $C=C(d)$. Moreover, from H\"older inequality and \eqref{eq:propSec2.2_2} we deduce that
\begin{align*}
\int_ {B_{R}(o)}f(t,x) \, dx &\leq C \int_0 ^t \int_ {B_{R}(o)} |u(t) - v(t)|^m \, dx\\
&\leq C \int_0^t vol(B_{R}(o))^{1-m}\left(\int_ {B_{R}(o)} |u(s) - v(s)| \right)^m \, ds\\
&\leq C  vol(B_{R}(o))^{1-m} \int_0^t \frac{vol(B_{\gamma R}(o))^m}{R^\frac{2m}{1-m}}s^{m/(1-m)} \, ds\\
&\leq C(\gamma) \frac{vol(B_{R}(o))}{R^\frac{2m}{1-m}}t^{1/(1-m)},
\end{align*}
and inserting  the last inequality into \eqref{eq:propSec2.2_3} and letting $R\to \infty$ we get the required conclusion.
\end{proof}
\bibliography{bibliografia_notes_on_PME}
\nocite{*}
\bibliographystyle{plain}

\end{document}